\documentclass[12pt]{amsart}
\newcommand{\deleted}[1]{}
\newcommand{\delete}[1]{}
\newcommand{\mynotes}[1]{}
\newcommand\notes[1]{}
\usepackage{txfonts}
\usepackage{amssymb}
\usepackage{eucal}
\usepackage{graphicx}
\usepackage{amsmath}
\usepackage{amscd}
\usepackage[all]{xy}
\usepackage[dvipdfm,  %pdflatex,pdftexÕâÀï¾ö¶¨ÔËÐÐÎļþµÄ·½Ê½²»Í¬
            pdfstartview=FitH,
            CJKbookmarks=true,
            bookmarksnumbered=true,
            bookmarksopen=true,
            colorlinks, %×¢Ê͵ô´ËÏîÔò½»²æÒýÓÃΪ²ÊÉ«±ß¿ò(½«colorlinksºÍpdfborderͬʱעÊ͵ô)
            pdfborder=001,   %×¢Ê͵ô´ËÏîÔò½»²æÒýÓÃΪ²ÊÉ«±ß¿ò
            linkcolor=blue,
            anchorcolor=blue,
            citecolor=blue
            ]{hyperref}

\usepackage{amsthm}
\usepackage{mathrsfs}
\usepackage{enumerate}
\usepackage{amsfonts,latexsym}
\usepackage{xspace}
\usepackage{epsfig}
\usepackage{float}
\usepackage{color}
\usepackage{fancybox}
\usepackage{colordvi}
\usepackage{multicol}
\usepackage{tikz}
\usepackage{wasysym}
\usepackage{xspace}
\usepackage[active]{srcltx} %SRC Specials for DVI Searching
\usepackage{mathtools} %  for left up and down index?
\usepackage{tabularx}
\usepackage{stfloats}

\newcommand\changed[1]{#1}
\changed{}

\newtheorem{theorem}{Theorem}[section]
\newtheorem{lemma}[theorem]{Lemma}

\newtheorem{coro}[theorem]{Corollary}

\newtheorem{prop}[theorem]{Proposition}
\theoremstyle{definition}
\newtheorem{defn}[theorem]{Definition}
\newtheorem{remark}[theorem]{Remark}
\newtheorem{exam}[theorem]{Example}

\deleted{\newtheorem{theorem}{Theorem}[section]
\newtheorem{prop-def}{Proposition-Definition}[section]
\newtheorem{coro-def}{Corollary-Definition}[section]

}

\newcommand{\nc}{\newcommand}
\renewcommand{\frak}{\mathfrak}
\newcommand{\efootnote}[1]{}

\nc{\mlabel}[1]{\label{#1}}  % Use this to suppress names
\nc{\mcite}[1]{\cite{#1}}  % Use this to suppress names
\nc{\mref}[1]{\ref{#1}}  % Use this to suppress names
\nc{\mbibitem}[1]{\bibitem{#1}} % Use this to show number
\renewcommand\geq{\geqslant}
\renewcommand\leq{\leqslant}

\renewcommand\bar[1]{\overline{#1}}

\nc{\mrm}[1]{{\rm #1}}

\nc{\Hom}{\mrm{Hom}}

\nc{\Ext}{\mrm{Ext}}

\nc{\End}{\mrm{End}}

\nc{\rad}{\mrm{rad}}

\nc{\Aut}{\mrm{Aut}}

\nc{\X}{X^{\bullet}}

\nc{\Y}{Y^{\bullet}}

\nc{\dcp}{the double centraliser property }

\nc{\ddcp}{the derived double centraliser property}

\nc{\add}{\mrm{add\,}}

\nc{\RHom}{\mrm{{\bf R}Hom}}

\nc{\OT}{\otimes^{\bf L}}

\nc{\thick}{\mrm{thick}}

\nc{\cone}{\mrm{cone}}

\nc{\proj}{\mrm{proj\,}}

\nc{\gl}{\mrm{gl}}

 \nc{\im}{\mrm{im}}

 \nc{\id}{\mrm{id}}

\nc{\redt}[1]{\textcolor{red}{#1}}

\nc{\jing}[1]{\textcolor{red}{Jing:#1}} %%when accepted, removed

\nc{\xing}[1]{\textcolor{purple}{Xing:#1}}

\begin{document}
\title[Complexes with the derived double centraliser property]{Complexes with the derived double centraliser property}
\author[Jin Zhang]{Jin Zhang}
\address{School of Mathematics and Statistics,
Lanzhou University, Lanzhou, 730000, P. R. China}
\email{zj\_10@lzu.edu.cn}

%\author[Alexander Zimmermann]{Alexander Zimmermann}
%\address{Alexander Zimmermann: Universit\'{e} de Picardie, D\'{e}partement de Math\'{e}matiques et LAMFA (UMR 7352 du CNRS), 33 rue St Leu, %F-80039 Amiens Cedex 1, France}
%\email{alexander.zimmermann@u-picardie.fr}

%========================================================================
\hyphenpenalty=8000

\begin{abstract}
In representation theory, the double centraliser property is an important property for a module (bimodule). It plays a fundamental role in many theories. In this paper, we extend this property to complexes in derived categories of finite dimensional algebras, under the name derived double centraliser property. Characterizations for complexes with the derived double centraliser property and (two-sided) tilting complexes in derived categories of hereditary algebras are given. In particular, all complexes with this property in the derived categories of lower triangular matrix algebras are classified.
\end{abstract}

\subjclass[2010]{16E05, 16E35, 16E60}

\keywords{Double centraliser property, complexes, derived categories, hereditary algebras}

\maketitle

\allowdisplaybreaks

%========================================================================

\section{Introduction}

Let $A$ be an algebra and $M$ an $A$-module with endomorphism algebra $B$. Then naturally
$M$ is an $A$-$B$-bimodule and $M$ is said to have the double centralizer property if left multiplication by elements in $A$ is an isomorphism $A\cong \End_{B^{op}}(M)$.

Modules with the double centralizer property occur frequently in algebra. Simple modules over central
simple algebras do satisfy the property, and more generally a Morita bimodule
satisfies the double centralizer property. Schur-Weyl duality (cf. \cite{Green}) is another
prominent occurrence. More recently it was shown (cf. \cite[Section 4.1]{Ringel}) that a tilting module over a finite dimensional algebra is another case. Only very recently a systematic classification of modules with the
double centralizer property was started by Crawley-Boevey, Ma, Rognerud and Sauter \cite{CMRS}
in the special case of lower triangular matrix algebras, i.e., hereditary algebras of type $A$.

Rickard's Morita theory for derived categories gives another case. Rickard showed \cite{R} that
if $A$ and $B$ are derived equivalent as algebras over a commutative ring $R$, both suppose to be projective over $R$, then
there is a complex $\X$ in the bounded derived category $D^b(A\otimes_RB^{op})$ of $A$-$B$-bimodules
such that $\X\otimes_B^{\mathbf L}-:D^-(B)\rightarrow D^-(A)$ is an equivalence of triangulated categories.
Such a complex is called a two-sided tilting complex and Rickard showed that $\X$ also has the double
centralizer property in the sense that right multiplication by elements in $B$ gives an isomorphism
$B\cong \End_{D^-(A)}(\X)^{op}$ and left multiplication by elements of $A$ gives an isomorphism
$A\cong \End_{D^-(B^{op})}(\X)$. We call this property under the name derived double centralizer property. See Definition \mref{def} for a more precise statement. Moreover, we point out that the notation of derived double centraliser property is different from the property of a complex $_A\X_B$ that $\RHom_A(\X, \X)^{op}\cong B$ and $\RHom_{B^{op}}(\X, \X)\cong A$ in the sense of Keller \cite{K}, Keller showed that such a complex preserves the Hochschild cohomologies of the algebras, however, complexes with the derived double centraliser property in our sense don't. A typical example we refer to work \cite{FM} by Fang and Miyachi. Note that, if $B$ is hereditary, then two-sided tilting and double centralizer proerty in the sense of Keller coincide. We characterize two-sided tilting complexes in the same way as complexes with the derived double centralizer property in this situation in our main Theorem \mref{1.1}.

The purpose of this paper is to study the derived double centralizer property systematically.
Since this seems to be vast, we start with the case when $A$ and $B$ are
finite dimensional algebras and in addition $B$ is hereditary. Note that we need to start with
a complex of bimodules rather than with a complex in $D^-(A)$ since  the endomorphism algebra
in the derived category of $A$ of such a complex does not give an action on the complex itself. In some
cases, it is possible to lift this to an action on the complex, but Keller showed that this is a highly
non trivial task \cite{Keller}.

Our main results make use of left approximations. If $\mathcal C$ is a Krull-Schmidt category and $\mathcal D$ is a full subcategory of $\mathcal C$, then a left $\mathcal D$-approximation of an object $X$ of $\mathcal C$ is a morphism $f: X\rightarrow D$ such that $D$ is an object of $\mathcal D$
and such that $\Hom_{\mathcal C}(f,D')$ is an epimorphism for any object $D'$ in $\mathcal D$. The approximation $f$ is left minimal if $gf=f$ for an endomorphism $g$ implies that $g$ is an isomorphism.
See Section~\mref{approximation} for more ample details.

We denote by $r_{\X}$ the right multiplication map, roughly speaking, for the precise definition of $r_{\X}$, we refer to Definition \mref{def}.

\begin{theorem}(Theorems \mref{charact} and \mref{t})\mlabel{1.1}
Let $A$ and $B$ be algebras. Let $\X$ be a bounded complex of $A$-$B$-bimodules. Assume that $B$ is hereditary and $r_{\X}:B\rightarrow\End_{D^b(A)}(\X)^{op}$ is an isomorphism. Then $\X$ has the derived double centraliser property (resp. is two-sided tilting) if and only if for any indecomposable projective $A$-module $Ae$,
\begin{enumerate}
\item
there is a unique $i\in \mathbb{Z}$ such that $\Hom_{D^b(A)}(Ae[i], \X)\neq 0$;
\item
there exists a minimal left $\add \X$-approximation sequence
$$Ae[i]\rightarrow \X_0\stackrel{g}\rightarrow \X_1$$
of $Ae[i]$ such that $Ae\cong H^{-(i+1)}(\cone(g))$ as $A$-modules (resp. which forms a distinguished triangle in $D^b(A)$).
\end{enumerate}
\end{theorem}

In \cite{MY}, Miyachi and Yekutieli studied the derived Picard group of a hereditary algebra $A$ by associating an automorphism of the Auslander-Reiten quiver of the derived category $D^b(A)$ to each two-sided tilting complex of $A$-bimodules. We give an intrinsic characterization of (two-sided) tilting complexes of hereditary algebras by approximation theory. Combining with complexes with the derived double centraliser property, the difference between them are reflected on the approximation sequences of the shifted of indecomposable projective modules.

\begin{prop}(Propositions \mref{dd} and \mref{two-tilting})
Let $A$ be a hereditary algebra and $\X$ a bounded complex of $A$-modules. Assume $\End_{D^b(A)}(\X)$ is hereditary.
Then $\X$ has the derived double centraliser property (resp. is isomorphic to a tilting complex in $D^b(A)$) if and only if for any indecomposable projective $A$-module $Ae$,
\begin{enumerate}
\item
there is a unique $i\in \mathbb{Z}$ such that
$\Hom_{D^b(A)}(Ae[i], \X)\neq 0;$
\item
there is an exact minimal left $\add H^{-i}(\X)$-approximation sequence
$$Ae\stackrel{f}\rightarrow X_0\rightarrow X_1\,\,\,\,  (\emph{resp.}\,\, Ae\stackrel{f}\rightarrow X_0\rightarrow X_1\rightarrow 0)$$
of $Ae$ such that $\ker f\in  \add H^{-(i+1)}(\X)$.
\end{enumerate}
\end{prop}

Let $A$ be a lower triangular matrix algebra. Crawley-Boevey, Ma, Rognerud and Sauter classified \cite{CMRS} the $A$-modules with the double centraliser property and related them to some combinatorial objects. In the present paper, we classified all bounded complexes of $A$-bimodules with the derived double centraliser property, see Theorem \mref{A_n}.

The paper is organized as follows: In Section \mref{basic}, we give the necessary notation and definitions. In Section \mref{definition}, the definition of complexes with the derived double centraliser property is provided and the related basic properties are also given. In Section \mref{characterizations}, we characterize the complexes with the derived double centraliser property and two-sided tilting complexes over hereditary algebras and describe their homologies. In section \mref{triangular}, complexes with the derived double centraliser property over lower triangular matrix algebras are classified.

\section{Notation and definitions}\mlabel{basic}

Let $K$ be a field. Throughout, all algebras are assumed to be finite dimensional $K$-algebras and all modules are assumed to be finite dimensional left modules, unless stated otherwise. Let $A$ be an algebra. $A^{op}$ denotes the opposite algebra of $A$, hence $A^{op}$-modules represent right $A$-modules. $A^e$ denotes the enveloping algebra $A\otimes_KA^{op}$. Denote by $A$-mod the category of all $A$-modules and by $\proj A$ the full subcategory of $A$-mod which contains all projective $A$-modules.

\subsection{Complexes and categories}
Let $A$ be an algebra. A complex $\X=(X^i, d_X^i)$ of $A$-modules, we simply write $A$-complex, is a sequence of $A$-modules $X^i$ and $A$-module morphisms $d_X^i: X^i\rightarrow X^{i+1}$ such that $d_X^{i+1}d_X^i=0$ for all $i\in\mathbb{Z}$.
 A morphism $f^{\bullet}: \X\rightarrow \Y$ of $A$-complexes $\X$ and $\Y$ is a collection of morphisms $f^i: X^i\rightarrow Y^i$ of $A$-modules such that $d_Y^if^i=f^{i+1}d_X^i$.  The complex $\X$ is called \textit{bounded above} if there is
a number $n$ such that $X^i=0$ for all $i>n$, \textit{bounded below} if there is
a number $m$ such that $X^i=0$ for all $i<m$, and \textit{bounded} if it is both bounded above and bounded below. $\X$ is called \textit{radical} if all its differentials are radical morphisms.
Assume $X^i=0$ for all $i<m$ and $i>n$ and $X^m\neq 0\neq X^n$, then the \textit{width} of $\X$ is defined as $n-m+1$. Note that, in this paper, we will not distinguish an $A$-module $X$ and the $A$-complex $\X$ which is concentrated in degree $0$ and $X^0=X$.

We denote by $C(A)$ the category of $A$-complexes, by $K(A)$ (resp. $K^b(A)$, $K^+(A)$, $K^-(A)$) the homotopy category of (resp. bounded, bounded above, bounded below) $A$-complexes,  and by $D(A)$ (resp. $D^b(A)$, $D^+(A)$, $D^-(A)$) the derived category of (resp. bounded, bounded above, bounded below) $A$-complexes.

%\subsection{Functors}
%Let $A, B$ and $C$ be algebras. Given a complex $\X$ of $A$-$B$-bimodules, we simply write $A$-$B$-complex, there are covariant triangle functors
%$$\RHom_A(\X, -): K(A\otimes_KC^{op})\rightarrow K(B\otimes_KC^{op})\,\,\, \textmd{and}\,\,\, \X\OT_B-: K(B\otimes_KC^{op})\rightarrow K(A\otimes_KC^{op}),$$
%and a contravariant triangle functor
%$\RHom_A(-, \X): K(A\otimes_KC^{op})\rightarrow K(C\otimes_KB^{op})$, they are calculated by taking total complex of the bicomplex induced by $\Hom(-, -)$ or $-\otimes-$. There are covariant derived functors
%$$\RHom_A(\X, -): D^+(A\otimes_KC^{op})\rightarrow D(B\otimes_KC^{op})\,\,\, \textmd{and}\,\,\, \X\OT_B-: D^-(B\otimes_KC^{op})\rightarrow D(A\otimes_KC^{op}),$$
%and a contravariant derived functor
%$\RHom_A(-, \X): D^-(A\otimes_KC^{op})\rightarrow D(C\otimes_KB^{op})$, they are calculated by taking injective resolutions of the covariant active variables, projective resolutions of the contravariant active variables and leaving the passive variables unresolved. For more details, we refer to \cite{CE, Z}.
%Moreover, for a complex $\Y\in D^-(B\otimes_KC^{op})$ and a complex $Z^{\bullet}\in D^+(A\otimes_KE^{op})$ with $E$ an algebra, there is an adjoint isomorphism
%$$\RHom_A(\X\OT_B\Y, Z^{\bullet})\cong \RHom_B(\Y, \RHom_A(\X, Z^{\bullet}))$$
%in $D(C\otimes_KE^{op})$.

\subsection{Derived equivalences and tilting complexes}
Let $A$ and $B$ be algebras. Recall, due to Rickard \cite{R0} and Keller \cite{K1}, that the following conditions are equivalent:
\begin{itemize}
  \item $D^b(A)$ and  $D^b(B)$ are equivalent as triangulated categories;
  \item $K^b(\proj A)$ and  $K^b(\proj B)$ are equivalent as triangulated categories;
  \item There is a complex $T^\bullet\in K^b(\proj A)$ such that $\End_{D^b(A)}(T^\bullet)^{op}\cong B$, $T^\bullet$ is self-orthogonal (i.e., $\Hom_{D^b(A)}(T^\bullet, T^\bullet[i])=0$ unless $i=0$), and $\add T^\bullet$ generates $K^b(\proj A)$ as triangulated category.
\end{itemize}
If one of above conditions holds, algebras $A$ and $B$ are said \textit{derived equivalent}. An $A$-complex satisfies the third condition is called a \textit{tilting complex}. Later,
in \cite{R}, he promoted this setting further. A complex $\Delta^\bullet\in D^b(A\otimes_KB^{op})$ is called \textit{two-sided tilting} if there is a complex $\Theta^\bullet\in D^b(B\otimes_KA^{op})$ such that $\Delta^\bullet\OT_B\Theta^\bullet\cong {_AA_A}$ and $\Theta^\bullet\OT_A\Delta^\bullet\cong {_BB_B}$. Then
$$\Theta^\bullet\OT_A-: D^b(A)\leftrightarrow D^b(B): \Delta^\bullet\OT_B-$$
are mutually inverse triangle equivalences. Moreover, it is shown that for any tilting $A$-complex $T^{\bullet}$ with $\End_{D^b(A)}(T^{\bullet})^{op}\cong B$, there is a two-sided tilting $A$-$B$-complex $\X$ such that $\X\cong T^{\bullet}$ in $D^b(A)$. Combining with the result in \cite{CPS}, a two-sided tilting $A$-$B$-complexes with homology concentrated in degree $0$ is isomorphic to an $A$-$B$-bimodule $T$ such that $_AT$ is tilting and $B\cong \End_A(T)^{op}$ canonically. Such an $A$-$B$-bimodule $T$ we simply called \textit{tilting}. For the knowledge of tilting module, we refer to the book \cite{H}.
\subsection{Approximations}

\mlabel{approximation}
Let $\mathcal{C}$ be a Krull-Schmidt category. Let $X$ be an object in $\mathcal{C}$. We say $X$ is \textit{basic} if it is isomorphic to a direct sum of pairwise non-isomorphic indecomposable objects in $\mathcal C$. Denote by $\add X$ the full subcategory of $\mathcal{C}$ which contains all direct summands of direct sums of copies of $X$.
Let $\mathcal{D}$ be a full subcategory of $\mathcal{C}$. Recall that a sequence
$$\xi: X\stackrel{f}\rightarrow D_0\rightarrow D_1\rightarrow \cdots\rightarrow D_n$$
of morphisms in $\mathcal{C}$ is called a \textit{left $\mathcal{D}$-approximation sequence} of $X$ if $D_i\in \mathcal{D}$,
$$\Hom_{\mathcal{C}}(\xi, D'): \Hom_{\mathcal{C}}(D_n, D')\rightarrow\cdots\rightarrow\Hom_{\mathcal{C}}(D_1, D')\rightarrow  \Hom_{\mathcal{C}}(D_0, D')\rightarrow \Hom_{\mathcal{C}}(X, D')$$
is exact and $\Hom_{\mathcal{C}}(f, D')$ is an epimorphism, for any $D'\in \mathcal{D}$. A morphism $f: X\rightarrow Y$ in $\mathcal{C}$ is called \textit{left minimal} if any morphism $g\in \End_{\mathcal{C}}(Y)$ with $gf=f$ is an isomorphism. The sequence $\xi$ is called a minimal left $\mathcal{D}$-approximation sequence if it is a left $\mathcal{D}$-approximation sequence and every morphism in $\xi$ is left minimal. If $n=0$, (minimal) left $\mathcal{D}$-approximation sequence is abbreviated to (minimal) left $\mathcal{D}$-approximation.
Dually, there is the notion of \textit{right minimal} and \textit{(minimal) right $\mathcal{D}$-approximation (sequence)}.

\section{Definition and basic properties of complexes with \ddcp}\mlabel{definition}

In this section, we will provide the definition of complex of bimodules with the derived double centraliser property and then some related basic results are given.

\subsection{Definition}
Let $A$ and $B$ be two algebras. For an $A$-$B$-bimodule $X$, the natural left multiplication algebra homomorphism and right multiplication algebra homomorphism defined respectively as:
$$l_X: A\rightarrow \End_{B^{op}}(X),\, a\mapsto (x\mapsto ax),\,\,\,\,\,r_X: B\rightarrow \End_A(X)^{op},\, b\mapsto (x\mapsto xb).$$
We say that $X$ has \textit{the double centraliser property} if both $l_X$ and $r_X$ are isomorphisms. Moreover, we say that $_AX$  has the double centraliser property if the bimodule $_AX_{\End_A(X)}$ has the double centraliser property. Note that $l_X$ is an $A$-bimodule isomorphism and $r_X$ is a $B$-bimodule isomorphism.

Let $\X$ be an $A$-$B$-complex. Then there are natural maps
$$l'_{\X}: A\rightarrow \End_{C(B^{op})}(\X),\, a\mapsto (l_{X^i}(a))_i,\,\,\,\,\,r'_{\X}: B\rightarrow \End_{C(A)}(\X)^{op},\, b\mapsto (r_{X^i}(b))_i.$$
Note that the differentials of $\X$ are $A$-$B$-bimodule morphisms, so $(l_{X^i}(a))_i\in \End_{C(B^{op})}(\X)$ and $(r_{X^i}(b))_i\in \End_{C(A)}(\X)^{op}$. It is easy to check that $l'_{\X}$ and $r'_{\X}$ are algebra homomorphisms.
Now let
$$l_{\X}: A\stackrel{l'_{\X}}\rightarrow \End_{C(B^{op})}(\X)\rightarrow \End_{K(B^{op})}(\X)\rightarrow \End_{D(B^{op})}(\X),$$
$$r_{\X}: B\stackrel{r'_{\X}}\rightarrow \End_{C(A)}(\X)^{op}\rightarrow \End_{K(A)}(\X)^{op}\rightarrow \End_{D(A)}(\X)^{op}$$
where the last two maps of $l_{\X}$ and $r_{\X}$ are the natural homotopic algebra quotients and the localisation algebra homomorphisms.

\begin{defn}\mlabel{def}
Let $A$ and $B$ be algebras. For an $A$-$B$-complex $\X$, we say that $\X$ has the derived double centraliser property if both $l_{\X}$ and $r_{\X}$ are isomorphisms. Moreover, for an $A$-complex $\X$, we say that $\X$ has the derived double centraliser property if there is an $A$-$B$-complex $\Y$ for some algebra $B$ such that $\X\cong \Y$ in $D(A)$ and $\Y$ as $A$-$B$-complex has the derived double centraliser property.
\end{defn}

Let $A$ and $B$ be algebras. Let $\X$ and $\Y$ be $A$-$B$-complexes. Suppose $s: \Y\rightarrow\X$ is a quasi-isomorphism of $A$-$B$-complexes, then there is a map
$$S: \End_{D(A)}(\X)^{op}\rightarrow \End_{D(A)}(\Y)^{op},\, \bar{a}/\bar{b}\mapsto \bar{as}/\bar{bs}$$
 where for any $a\in \End_{C(A)}(\X)$, $\bar{a}$ represents the corresponding element of $a$ in $K(A)$.
Note that we used left roofs to represent morphisms in derived categories. It is easy to check that $S$ is an algebra isomorphism.

Next, we show that $r_{\Y}=r_{\X}S$. Recall that $r_{\X}(b)=\bar{r'_{\X}(b)}/1$ for $b\in B$. So we only need to prove that $\bar{r'_{\Y}(b)}/1=\bar{r'_{\X}(b)s}/\bar{s}$ in $\End_{D(A)}(\Y)^{op}$. This is due to the following commutative diagram:
\begin{displaymath}
\xymatrixrowsep{0.2in}
\xymatrixcolsep{0.5in}
\xymatrix{
 & \X\ar[d]^{1}& \\
\Y \ar[ur]^{\bar{r'_{\X}(b)s}}\ar[r]_{\bar{r'_{\X}(b)s}} \ar[dr]_{\bar{r'_{\Y}(b)}} &\X& \Y\ar[ul]_{\bar{s}}\ar[l]^{\bar{s}}\ar[dl]^{1}\\
 &\Y\ar[u]^{\bar{s}}&
}
\end{displaymath}
where $\bar{r'_{\X}(b)s}= \bar{s}\bar{r'_{\Y}(b)}$ because $s$ is a morphism of $B^{op}$-complexes.

So in this case, $r_{\X}$ is an isomorphism if and only if $r_{\Y}$ is an isomorphism. Similarly, $l_{\X}$ is an isomorphism if and only if $l_{\Y}$ is an isomorphism. Conclusively,
\begin{lemma}\mlabel{quasi}
Let $A$ and $B$ be two algebras. Let $\X$ and $\Y$ be two $A$-$B$-complexes. If $\X$ and $\Y$ are quasi-isomorphic as $A$-$B$-complexes, assume $S: \End_{D(A)}(\X)^{op}\rightarrow \End_{D(A)}(\Y)^{op}$ is the isomorphism induced by the quasi-isomorphism, then $r_{\Y}=r_{\X}S$. In this case, $\X$ has the derived double centraliser property if and only if so does $\Y$.
\end{lemma}

\begin{remark}\mlabel{bimodule}
\begin{enumerate}
\item
 Using Lemma \mref{quasi}, a module or a bimodule has the double centraliser property if and only if it has \ddcp.
\item
  It is shown by Rickard in \cite{R} that two-sided tilting complexes have the derived double centraliser property.
\item
Yekutieli showed \cite{Ye} dualizing complexes have \ddcp.
\item
More generally, as mention before, $A$-$B$-complexes $\X$ which satisfy $\RHom_A(\X, \X)^{op}\cong B$ and $\RHom_{B^{op}}(\X, \X)\cong A$ have the double centraliser property in the sense of Keller \cite{K}.
\end{enumerate}
\end{remark}

\subsection{Action on complexes}
Let $A$, $B$ and $C$ be algebras, $\X$ an $A$-$B$-complex, and $\Y$ an $A$-$C$-complex. Then $\Hom_{D(A)}(\X, \Y)$ has a $B$-$C$-bimodule structure. For $f\in \Hom_{D(A)}(\X, \Y), a\in A$ and $b\in B$, the left $B$-module structure is given by $bf=fr_{\X}(b)$ and the right $C$-module structure is given by $fc=r_{\Y}(c)f$.
In this subsection, we collect some results around this topic.
%One has to distinguish that, if $\End_{D(A)}(\X)\cong B$, $\Hom_{D(A)}(\X, \Y)$ has another $B$-module structure given by composite of

\begin{lemma}\mlabel{bi}
Let $A$ and $B$ be two algebras and $\X$ an $A$-$B$-bimodule complex. The following conditions are equivalent:
\begin{enumerate}
\item
$r_{\X}$ is an isomorphism,
\item
there is a $B$-module isomorphism $B\rightarrow \End_{D(A)}(\X)$,
\item
there is a $B^{op}$-module isomorphism $B\rightarrow \End_{D(A)}(\X)$.
\end{enumerate}
\end{lemma}
\begin{proof}
The proof is similar to the bimodule version in \cite[Lemma 2.3]{ZL}.
For any $b, b'\in B$, $r_{\X}(bb')=r_{\X}(b')r_{\X}(b)=br_{\X}(b')$. Then $r_{\X}$ is a $B$-module morphism. Similarly, $r_{\X}$ is also a $B^{op}$-module morphism. Hence if $r_{\X}$ is an isomorphism, then it is a $B$-module isomorphism and is also a $B^{op}$-module isomorphism.
Suppose there is a $B$-module isomorphism $F: B\rightarrow \End_{D(A)}(\X)$. For any $0\neq b\in B$, $F(b)=bF(1)=F(1)r_{\X}(b)\neq 0$, hence $r_{\X}(b)\neq 0$, that is, $r_{\X}$ is injective. Moreover, since $B\cong\End_{D(A)}(\X)$ as $B$-modules, $\dim_K B=\dim_K\End_{D(A)}(\X)$, then we have $r_{\X}$ is an isomorphism.
The proof of right version is similar.
\end{proof}

\begin{lemma}\mlabel{r}
Let $A$ and $B$ be algebras and $\X$ a bounded $A$-$B$-complex. Let $b\in B$ and $f:B\rightarrow Bb$ the canonical epimorphism. Then $\X\OT_BBb= 0$ in $D^b(A)$ if and only if $\X\OT_B f=0$ in $D^b(A)$. Moreover, in this case, $r_{\X}(b)=0$.
\end{lemma}
\begin{proof}
If $\X\OT_BBb= 0$ in $D^b(A)$, then, obviously, $\X\OT_B f: \X\rightarrow \X\OT_BBb$ is zero morphism in $D^b(A)$.
Conversely, suppose that $\X\OT_B f: \X\rightarrow \X\OT_BBb$ is zero morphism in $D^b(A)$. Let $\Y$ be a projective resolution of $A$-$B$-complex $\X$. Then $\X\OT_B f=0$ in $D^b(A)$ implies $\Y\OT_B f=0$ in $K^-(A)$. So there is a homotopy $F:\Y\rightarrow \Y\OT_BBb[-1]$ for $\Y\OT_B f$. Assume $Y^i=0$ for $i>k$ and $Y^k\neq 0$. Then
$Y^k\otimes_Bf=d^{k-1}F^i$, where $d$ is the differential of $\Y\OT_BBb$. Note that $Y^k\otimes_Bf$ is surjective, then so is $d^{k-1}$. Since $\Y\OT_BBb$ is a complex of projective $A$-modules,  $d^{k-1}$ is split as a morphism of $A$-modules. Iteratively, we have $\Y\OT_B Bb=0$ in $K^-(A)$, then $\X\OT_BBb= 0$ in $D^b(A)$.

Let $h: B\stackrel{f}\rightarrow Bb\stackrel{g}\hookrightarrow B$ with $g$ the canonical embedding. To prove $r_{\X}(b)=0$, we only need to show that $r_{\X}(b)=0$ if and only if $\X\OT_B h=0$, since then $\X\OT_B h=\X\OT_B gf=(\X\OT_B g)(\X\OT_B f)=0$ in $D^b(A)$, so $r_{\X}(b)=0$.
Indeed, $\X$ is quasi-isomorphic to $\X\OT_BB$ as $A$-$B$-complexes, by Lemma \mref{quasi}, we have $r_{\X}(b)=0$ if and only if $r_{\X\OT_BB}(b)=0$. We point out that $r_{\X\OT_BB}(b)=\X\OT_B h$.
\end{proof}

\begin{exam}\mlabel{exam}
Let $A$ and $B$ be algebras and $\X$ a bounded $A$-$B$-complex. Note that, for $b\in B$, in general, $r_{\X}(b)=0$ could not imply $\X\OT_BBb= 0$ in $D^b(A)$.

Set $A=B=K(x)/(x^2)$ and $\X$ the $A^e$-complex concentrated in degrees $0$ and $1$ as $A^e\stackrel{r}\rightarrow A^e$, where $r$ is given by $r(a)=ax$ for $a\in A^e$. Then we can verify that $r'_{\X}(x)$ is homotopic to zero, hence $r_{\X}(x)=0$. However, $\X\OT_AAx\cong A^ex\oplus A^ex[-1]$ in $D^b(A)$ which is not the zero object (note that a two-term complex is homotopic to zero if and only if the connected morphism of the two terms is an isomorphism).
\end{exam}

\begin{lemma}\mlabel{direct}
Let $A$ and $B$ be two algebras and $\X$ a bounded $A$-$B$-complex. Assume $r_{\X}$ is an isomorphism and $\{e_1, e_2, \cdots, e_n\}$ is a complete set of primitive orthogonal idempotents of $B$, then $\oplus_{i=1}^n(\X\OT_BBe_i)$ is a decomposition of $\X$ into indecomposable objects in $D^b(A)$.
\end{lemma}
\begin{proof}
Note that  $\X\cong \X\OT_BB\cong \X\OT_B(\oplus_{i=1}^n Be_i)\cong\oplus_{i=1}^n(\X\OT_BBe_i)$ in $D^b(A)$. Since $r_{\X}$ is an isomorphism, the number of indecomposable direct summands of $\X$ in $D^b(A)$ is equal to the cardinal of a complete set of primitive orthogonal idempotents of $B$. So we only need to prove that, for each $e_i$, $\X\OT_BBe_i\neq 0$ in $D^b(A)$. But this directly follows from Lemma \mref{r}.
\end{proof}

\begin{lemma}\mlabel{left-minimal}
Let $A$ and $B$ be algebras. Let $\X$ be a bounded $A$-$B$-complex. Assume that $r_{\X}$ is an isomorphism. A sequence  $\xi: \Y\rightarrow \X_0\rightarrow \X_1$ in $D^b(A)$ is a minimal left $\add \X$-approximation sequence of $\Y$ if and only if
$$\Hom_{D^b(A)}(\xi, \X): \Hom_{D^b(A)}(\X_1, \X)\rightarrow\Hom_{D^b(A)}(\X_0, \X)\rightarrow \Hom_{D^b(A)}(\Y, \X)$$
is a minimal projective presentation of $B^{op}$-module $\Hom_{D^b(A)}(\Y, \X)$.
\end{lemma}
\begin{proof}
 We abbreviate $(-)^*=\Hom_{D^b(A)}(-, \X)$. By definition, we may assume $\xi: \Y\stackrel{f}\rightarrow \X_0\stackrel{h}\rightarrow \X_1$ is a left $\add \X$-approximation sequence of $\Y$, $\Hom_{D^b(A)}(\xi, \X)$ is exact and $f^*$ is an epimorphism. Then we only need to prove that $f$ (resp. $h$) is left minimal if and only if $f^*$ (resp. $h^*$) is right minimal.  By Lemma \mref{bi}, $r_{\X}:B\rightarrow \End_{D^b(A)}(\X)$ is a $B^{op}$-module isomorphism, then there are mutually inverse dualities: $$\Hom_{D^b(A)}(-, \X): \add \X\leftrightarrow \proj B^{op}: \X\OT_B\Hom_{B^{op}}(-, B).$$
So $h$ is left minimal if and only if $h^*$ is right minimal.
 Suppose that $f^*$ is right minimal. For any $g\in \End_{D^b(A)}(\X_0)$ with $gf=f$, then $f^*g^*=f^*$, so $g^*$ is an isomorphism. Since $(-)^*: \add \X\rightarrow \proj B^{op}$ is a duality, $g$ is an isomorphism. Conversely, suppose that $f$ is left minimal. For any $g'\in \End_{B^{op}}((\X_0)^*)$ with $f^*g'=f^*$. Since $\End_{D^b(A)}(\X_0)\cong\End_{B^{op}}((\X_0)^*)$ canonically, there is $g\in\End_{D^b(A)}(\X_0)$ such that $g'=g^*$. Then we have $(gf)^*=f^*$.
Note that $r_{\X}$ is an isomorphism, then $\Hom_{D^b(A)}(\Y, \X_0)\cong\Hom_{B^{op}}((\X_0)^*, (\Y)^*)$ canonically. We have $gf=f$, hence $g$ is an isomorphism, then so is $g'$.
\end{proof}

\subsection{Several properties related to \ddcp}

\begin{lemma}\cite[Lemma 1.1]{KY}\mlabel{KY}
Let $A$ and $B$ be two algebras. Let $X$ and $Y$ be two $A$-$B$-bimodules. If $X\cong Y$ as $A$-modules and $r_X$ and $r_Y$ are isomorphisms, then there is $\beta\in \Aut(B)$ such that $X\cong Y_{\beta}$ as $A$-$B$-bimodules.
\end{lemma}

\begin{prop}\mlabel{two-sided}
Let $A$ and $B$ be two algebras and $\X$ an $A$-$B$-complex. If $\X$ is isomorphic to a tilting complex in $D^b(A)$ and $r_{\X}$ is an isomorphism, then $\X$ is two-sided tilting.
\end{prop}
\begin{proof}
Since $\X$ is isomorphic to a tilting complex in $D^b(A)$ and $r_{\X}: B\rightarrow \End_{D^b(A)}(\X)^{op}$ is an isomorphism, there is a two-sided tilting $A$-$B$-complex $\bar{\X}$ such that $\bar{\X}\cong \X$ in $D^b(A)$, and there is a two-sided tilting $B$-$A$-complex $\Y$ such that $\bar{\X}\OT_B\Y\cong {_AA_A}$ and $\Y\OT_A\bar{\X}\cong {_BB_B}$. Then we have $\Y\OT_A \X\cong B$ in $D^b(B)$, hence $H^0(\Y\OT_A \X)\cong B$ as $B$-modules. Note that  $H^0(\Y\OT_A \X)$ is a $B$-bimodule. Then we have the following $B$-module isomorphisms:
\begin{equation*}
\begin{aligned}
B \stackrel{r_{\X}}\rightarrow &\End_{D^b(A)}(\X)\cong \Hom_{D^b(A)}(\X, \bar{\X}\OT_B\Y\OT_A\X)\cong\Hom_{D^b(A)}(\X, \RHom_B(\Y, \Y\OT_A\X))\\
&\cong \End_{D^b(B)}(\Y\OT_A \X)\cong \End_B(H^0(\Y\OT_A \X)).
\end{aligned}
\end{equation*}
By Lemma \mref{bi} or \cite[Lemma 2.3]{ZL}, $r_{H^0(\Y\OT_A \X)}$ is an isomorphism. Note that $_BB_B$ has the double centraliser property, by Lemma \mref{KY}, there is $\sigma\in \Aut(B)$ such that $H^0(\Y\OT_A \X)\cong B_\sigma$ as $B$-bimodules. Then $\Y\OT_A \X\cong B_\sigma$ in  $D^b(B^e)$. Hence $\X\cong \bar{\X}\OT_B\Y\OT_A \X\cong\bar{\X}\OT_BB_\sigma$ in  $D^b(A\otimes_kB^{op})$. Note that both $\bar{\X}$ and $B_\sigma$ are two-sided tilting complexes, by \cite[Proposition 4.1]{R}, $\X$ is a two-sided tilting $A$-$B$-complex.
\end{proof}

For algebras $A$ and $B$, let $\lambda(A, B)$ be the number of all bounded $A$-$B$-complexes, up to isomorphism and shift in $D^b(A\otimes_kB^{op})$, which have \ddcp.

\begin{prop}\mlabel{out}
Let $A$ and $B$ be two algebras and $\X$ a bounded $A$-$B$-complex. Let $\sigma\in \emph{Aut}(B)$.
\begin{enumerate}
\item
$\X$ has the derived double centraliser property if and only if so does $\X\OT_BB_{\sigma}$.
\item
Assume $r_{\X}$ is an isomorphism. Then $\X\OT_BB_{\sigma}\cong \X$ in $D^b(A\otimes_KB^{op})$ if and only if $\sigma$ is inner.
\end{enumerate}
So both $|\emph{Out}(A)|$ and $|\emph{Out}(B)|$ divide $\lambda(A, B)$ if $\lambda(A, B)$ is finite.
\end{prop}
\begin{proof}
$(1)$. Since $\X\OT_BB_{\sigma}\OT_BB_{\sigma^{-1}}\cong\X$ in $D^b(A\otimes_KB^{op})$, we just prove the necessity.
Assume $\X$ has the derived double centraliser property. Note that $\sigma^{-1}: B_{\sigma}\rightarrow{_{\sigma^{-1}}B}$ and $\sigma: B\rightarrow{_{\sigma}B_{\sigma}}$ are $B$-bimodule isomorphisms. Then we have the following $A$-bimodule isomorphisms:
\begin{equation*}
\begin{aligned}
\End_{D^b(B^{op})}(\X\OT_BB_{\sigma})&\cong H^0\RHom_{B^{op}}(\X\OT_BB_{\sigma},\X\OT_BB_{\sigma})\\
&\cong H^0\RHom_{B^{op}}(\X, \RHom_{B^{op}}(B_{\sigma},\X\OT_BB_{\sigma})\\
&\cong H^0\RHom_{B^{op}}(\X, \RHom_{B^{op}}({_{\sigma^{-1}}B},\X\OT_BB_{\sigma})\\
&\cong H^0\RHom_{B^{op}}(\X, \X)\cong A,
\end{aligned}
\end{equation*}
and $B$-bimodule isomorphisms: $\End_{D^b(A)}(\X\OT_BB_{\sigma})\cong \End_{D^b(B)}(B_{\sigma})\cong {_{\sigma}B_{\sigma}}\cong B$.
By Lemma \mref{bi} and its dual, $l_{\X\OT_BB_{\sigma}}$ and $r_{\X\OT_BB_{\sigma}}$ are isomorphisms, hence $\X\OT_BB_{\sigma}$ has the derived double centraliser property.

$(2)$. Let $P(\X)$ be a projective resolution of $A$-$B$-complex $\X$. Since $\X\OT_BB_{\sigma}\cong \X$ in $D^b(A\otimes_KB^{op})$ if only if  $P(\X)\OT_BB_{\sigma}\cong P(\X)$ in $K^-(A\otimes_KB^{op})$, we only need to prove that $P(\X)\OT_BB_{\sigma}\cong P(\X)$ in $K^-(A\otimes_KB^{op})$ if and only if $\sigma$ is inner.

Suppose $\sigma\in \Aut(B)$ is inner. Assume for any $b\in B$, $\sigma(b)=z^{-1}bz$ for an invertible element $z\in B$.  Then $r'_{\X}(z):P(\X)\rightarrow P(\X)\OT_BB_{\sigma}$ is an isomorphism in $C^-(A\otimes_KB^{op})$, so $P(\X)\OT_BB_{\sigma}\cong P(\X)$ in $K^-(A\otimes_KB^{op})$.

Coversely, suppose that $P(\X)\OT_BB_{\sigma}\cong P(\X)$ in $K^-(A\otimes_KB^{op})$. By \cite[pp. 112 (a)]{HX}, there is a radical $A$-$B$-complex $\Y$ such that $P(\X)\cong\Y$ in $K^-(A\otimes_KB^{op})$. Note that $\Y$ is also a complex of projective $A\otimes_KB^{op}$-modules. Then $\Y\OT_BB_{\sigma}\cong \Y$ in $K^-(A\otimes_KB^{op})$. It is easy to check that $\Y\OT_BB_{\sigma}$ is also radical. By \cite[pp. 113 (b)]{HX}, $\Y\OT_BB_{\sigma}\cong \Y$ in $C^-(A\otimes_KB^{op})$.

Let $F': \Y\rightarrow \Y\OT_BB_{\sigma}$ be an isomorphism in $C^-(A\otimes_KB^{op})$ and $F$ be the corresponding isomorphism of $F'$ in $D^-(A\otimes_KB^{op})$.
Note that $\X$ is quasi-isomorphic to $\Y$, by Lemma \mref{quasi}, $r_{\X}$ is an isomorphism if and only if so is $r_{\Y}$. We identify $\Y\OT_BB_{\sigma}$ with $\Y$ in $D^-(A)$. Then $F=r_{\Y}(z)\in \End_{D^-(A)}(\Y)$ for some $z\in B$. As $F$ is an isomorphism in $\End_{D^-(A)}(\Y)$, $z$ is invertible in $B$. Since $F'$ is a morphism of $B^{op}$-complexes, then for any $b\in B$, $F'r'_{\Y}(b)=r'_{\Y\OT_BB_{\sigma}}(b)F'$. Hence $Fr_{\Y}(b)=r_{\Y\OT_BB_{\sigma}}(b)F$. So $r_{\Y}(z)r_{\Y}(b)=r_{\Y\OT_BB_{\sigma}}(b)r_{\Y}(z)=r_{\Y}(\sigma(b))r_{\Y}(z)$. As $r_{\Y}$ is an algebra isomorphism, then $\sigma(b)=zbz^{-1}$. We are done.
\end{proof}

We record the following two results which will be used in the sequel.
\begin{prop}\cite[Proposition 2.3]{RZ}\mlabel{RZ}
Let $A$ and $B$ be two algebras. Let $\X$ and $\Y$ be two-sided tilting $A$-$B$-complexes. Then $\X\cong \Y$ in $D^b(B^{op})$ if and only if there is $\alpha\in\Aut(A)$ such that $\X\cong {_\alpha A}\OT_A\Y$ in $D^b(A\otimes_KB^{op})$.
\end{prop}

\begin{prop}\cite[Proposition 2.1]{AS}\mlabel{AS}
Let $A$ be an algebra and $M$ an $A$-module. Then $M$ has \dcp if and only if there exists an exact sequence
$$0\rightarrow A\stackrel{f}\rightarrow M_0\rightarrow M_1$$
such that $M_0, M_1\in$ \emph{add} $M$ and $f:A\rightarrow M_0$ is a left \emph{add} $M$-approximation.
\end{prop}

\section{For hereditary algebras}\mlabel{characterizations}

In this section, explicit characterizations for complexes with the double centraliser property and two-sided tilting complexes over hereditary algebras are given. Hereditary algebras is a well-studied class of algebras with a lot of applications. The structure theory appears in categorifications of many mathematical objects such as cluster algebras and Lie algebras.

Let $A$ be a hereditary algebra. We will frequently use one of the following fundamental equivalent characterizations for hereditary algebras: $(1)$ $\gl(A)\leq 1$, where $\gl(A)$ denotes the global dimension of $A$; $(2)$ Any submodule of a projective $A$-module is still projective. Moreover, it is known that indecomposable objects in $D^b(A)$ is isomorphic to those of the form $X[i]$, for some indecomposable $A$-module $X$ and $i\in \mathbb{Z}$. Given another $A$-module $Y$ and $j\in\mathbb{Z}$, $\Hom_{D^b(A)}(X[i], Y[j])$ is isomorphic to $\Hom_A(X, Y)$ if $j=i$, to $\Ext^1_A(X, Y)$ if $j=i+1$, or to $0$ otherwise. Meanwhile, for a bounded $A$-complex $\X$, $\X\cong\oplus_{i\in\mathbb{Z}} H^i(\X)[-i]$ in $D^b(A)$, see \cite{H}.

\subsection{Characterizations I}

In this subsection, we will characterize $A$-$B$-complexes with the derived double centraliser property and two-sided tilting $A$-$B$-complexes where $A$ and $B$ are algebras and $B$ is hereditary.

\begin{lemma}\mlabel{app}
Let $A$ and $B$ be algebras. Let $\X$ be a bounded $A$-$B$-complex. Assume that $B$ is hereditary and $r_{\X}:B\rightarrow\End_{D^b(A)}(\X)^{op}$ is an isomorphism. Then for any bounded $A$-complex $\Y$, there is a minimal left $\add \X$-approximation sequence
$$\Y\rightarrow \X_0\stackrel{g}\rightarrow \X_1$$
 of $\Y$ such that $\cone(g)\cong \RHom_{B^{op}}(\Hom_{D^b(A)}(\Y, \X), \X)[1]$ in $D^b(A)$.
\end{lemma}
\begin{proof}
Since $r_{\X}:B\rightarrow\End_{D^b(A)}(\X)$ is a $B^{op}$-module isomorphism by Lemma \mref{bi} and $B$ is hereditary, let
\begin{equation}\tag{\dag}
0\rightarrow\Hom_{D^b(A)}(\X_1, \X)\rightarrow\Hom_{D^b(A)}(\X_0, \X)\rightarrow\Hom_{D^b(A)}(\Y, \X)\rightarrow 0
\end{equation}
 be a minimal projective presentation of the $B^{op}$-module $\Hom_{D^b(A)}(\Y, \X)$ with $\X_0, \X_1\in \add \X$. By Lemma \mref{left-minimal}, we have a minimal left $\add \X$-approximation sequence
 $\Y\rightarrow \X_0\stackrel{g}\rightarrow \X_1$ of $\Y$ in $D^b(A)$.
 Applying $\RHom_B(-, \X)$ to the distinguished triangle in $D^b(B^{op})$ induced by the exact sequence (\dag), we get a distinguished triangle
  \begin{equation*}
 \begin{aligned}
& \RHom_{B^{op}}(\Hom_{D^b(A)}(\Y, \X), \X)\rightarrow\RHom_{B^{op}}(\Hom_{D^b(A)}(\X_0, \X), \X) \rightarrow \\
&\RHom_{B^{op}}(\Hom_{D^b(A)}(\X_1, \X), \X)\rightarrow \RHom_{B^{op}}(\Hom_{D^b(A)}(\Y, \X), \X)[1]
\end{aligned}
\end{equation*}
in $D^b(A)$. Since $r_{\X}$ is an isomorphism, the endofunctor $\RHom_B(\Hom_{D^b(A)}(-, \X), \X)$ on $\add \X$ is an identity. Then we get the following  distinguished triangle
$$\RHom_{B^{op}}(\Hom_{D^b(A)}(\Y, \X), \X)\rightarrow \X_0\stackrel{g}\rightarrow\X_1\rightarrow \RHom_{B^{op}}(\Hom_{D^b(A)}(\Y, \X), \X)[1].$$
Hence $\cone(g)\cong \RHom_{B^{op}}(\Hom_{D^b(A)}(\Y, \X), \X)[1]$ in $D^b(A)$.
\end{proof}

\begin{theorem}\mlabel{charact}
Let $A$ and $B$ be algebras. Let $\X$ be a bounded $A$-$B$-complex. Assume that $B$ is hereditary and $r_{\X}:B\rightarrow\End_{D^b(A)}(\X)^{op}$ is an isomorphism. Then $\X$ has the derived double centraliser property if and only if for any indecomposable projective $A$-module $Ae$,
\begin{enumerate}
\item
there is a unique $i\in \mathbb{Z}$ such that $\Hom_{D^b(A)}(Ae[i], \X)\neq 0$;
\item
there exists a minimal left $\add \X$-approximation sequence
$$Ae[i]\rightarrow \X_0\stackrel{g}\rightarrow \X_1$$
of $Ae[i]$ in $D^b(A)$ such that $Ae\cong H^{-(i+1)}(\cone(g))$ as $A$-modules.
\end{enumerate}
\end{theorem}
\begin{proof}
Suppose that $\X$ has \ddcp. Then for any indecomposable $A$-module $Ae$ with $e$ an idempotent of $A$, by Lemma \mref{direct}, $eA\OT_A\X$ is an indecomposable direct summand of $\X$ in $D^b(B^{op})$. Since $B$ is hereditary, there exists a unique $i\in \mathbb{Z}$ such that $eA\OT_A\X\cong H^{-i}(eA\OT_A\X)[i]$ in $D^b(B^{op})$. Then $\Hom_{D^b(A)}(Ae[i], \X)\cong H^{-i}\RHom_A(Ae, \X)\cong H^{-i}(eA\OT_A\X)\cong eA\OT_A\X[-i]\neq 0$ in $D^b(B^{op})$. Consider the $A$-complex $Ae[i]$, by Lemma \mref{app}, there is a minimal left $\add \X$-approximation sequence
$$Ae[i]\rightarrow \X_0\stackrel{g}\rightarrow \X_1$$
 of $Ae[i]$ such that $\cone(g)\cong \RHom_{B^{op}}(\Hom_{D^b(A)}(Ae[i], \X), \X)[1]$ in $D^b(A)$. Then
 \begin{equation*}
 \begin{aligned}
H^{-(i+1)}(\cone(g))&\cong H^{-i}(\RHom_{B^{op}}(\Hom_{D^b(A)}(Ae[i], \X), \X))\cong H^{-i}(\RHom_{B^{op}}(eA\OT_A\X[-i], \X))\\
&\cong H^0(\RHom_{A^{op}}(eA, \RHom_{B^{op}}(\X, \X))\cong H^0(\RHom_{B^{op}}(\X, \X)\OT_AAe)\\
&\cong H^0(\RHom_{B^{op}}(\X, \X))\otimes_AAe\cong \Hom_{D^b(B^{op})}(\X, \X)\otimes_AAe\cong Ae
\end{aligned}
\end{equation*}
as $A$-modules.

Conversely, let $A\cong \oplus_jAe_j$ be a decomposition of $_AA$ into indecomposables $Ae_j$ with $e_j$ idempotents of $A$. Assume for any indecomposable projective $A$-module $Ae_j$, there is a unique $n(j)\in \mathbb{Z}$ such that $\Hom_{D^b(A)}(Ae[n(j)], \X)\neq 0$. Then, since $B$ is hereditary,
 \begin{equation*}
 \begin{aligned}
\RHom_A(Ae_j, \X)\cong \oplus_{k\in\mathbb{Z}}H^k(\RHom_A(Ae_j, \X))[-k]&\cong \oplus_{k\in\mathbb{Z}}\Hom_{D^b(A)}(Ae_j[-k], \X)[-k]\\
&\cong \Hom_{D^b(A)}(Ae_j[n(j)], \X)[n(j)]
\end{aligned}
\end{equation*}
in $D^b(B^{op})$. So $\X\cong \RHom_A(A, \X)\cong \RHom_A(\oplus_jAe_j, \X)\cong \oplus_j \Hom_{D^b(A)}(Ae_j[n(j)], \X)[n(j)]$ in $D^b(B^{op})$. By hypothesis, there exists a minimal left $\add \X$-approximation sequence
$$Ae_j[n(j)]\rightarrow \X_0\stackrel{g}\rightarrow \X_1$$
of $Ae[n(j)]$ in $D^b(A)$ such that $Ae_j\cong H^{-(n(j)+1)}(\cone(g))$ as $A$-modules. Combining with Lemma \mref{app},
 \begin{equation*}
 \begin{aligned}
A&\cong \oplus_jAe_j\cong \oplus_j H^{-n(j)}(\RHom_{B^{op}}(\Hom_{D^b(A)}(Ae_j[n(j)], \X), \X))\\
&\cong\Hom_{D^b(B^{op})}(\oplus_j \Hom_{D^b(A)}(Ae_j[n(j)], \X)[n(j)], \X)\\
&\cong \Hom_{D^b(B^{op})}(\X, \X)
\end{aligned}
\end{equation*}
as $A$-modules. By Lemma \mref{bi}, $l_{\X}$ is an isomorphism. So $\X$ has \ddcp.
\end{proof}

\begin{remark}
We point out that, when the $A$-$B$-complex $\X$ is concentrated as a bimodule, the characterization Theorem \mref{charact} coincides with the characterization Proposition \mref{AS} by Auslander and Solberg.
\end{remark}

Now, we are going to characterize two-sided tilting complexes.
Before we give the first lemma, let us recall the generating subcategories of triangulated categories. Let $A$ be an algebra and  $X$ be an object in $K^b(A)$. Denote by $\thick(X)$ the smallest triangulated full subcategory of $K^b(A)$ which contains $\add X$ and is closed under taking direct summand.
%Note that if $X$ is an $A$-module, $\thick(X)$ is just the triangulated full subcategory of $K^b(A)$ whose objects are isomorphic to bounded complexes with components in $\add X$.
Note that if $X$ is an object of $K^b(\proj A)$, then $X$ generates $K^b(\proj A)$ if and only if $_AA\in \thick(X)$.

The following lemma is implicitly in Rickard's Morita theory for derived categories.

\begin{lemma}\mlabel{faith}
Let $A$ and $B$ be algebras. Let $\X$ be a bounded $A$-$B$-complex. If $\RHom_{B^{op}}(\X, \X)\cong A$ in $D^b(A)$, then $$F:=\RHom_A(-, \X): K^b(\proj A)\leftrightarrow \thick(\X_B): \RHom_{B^{op}}(-, \X)=:G$$
are mutually inverse triangle dualities.
\end{lemma}
\begin{proof}
First, we prove that $F(\Y)$ is an object of $\thick(\X_B)$, for any object $\Y$ of $K^b(\proj A)$. We prove this by induction on the width $t(\Y)$ of $\Y$. When $t(\Y)=1$, this means that $\Y$, up to shift, is isomorphic to a projective $A$-module, then $F(\Y)\in \add \X$, hence $F(\Y)\in\thick(\X_B)$. Assume the assertion holds for $t(\Y)<n$. Let $t(\Y)=n$. We may
assume without loss of generality that the non-zero components of $\Y$ are exactly in degrees between $1$ and $n$. Let $\Y_{<n}$ be the brutal truncation of $\Y$ at degree $n$. Then there is a distinguished triangle $\Y_{<n}[-1]\rightarrow Y^{n}[-n]\rightarrow \Y\rightarrow \Y_{<n}$  in $K^b(\proj A)$. Applying $F$ to it we have a distinguished triangle $F(\Y_{<n})\rightarrow F(\Y)\rightarrow F(Y^{n}[-n])\rightarrow F(\Y_{<n})[1]$ in $K^b(B^{op})$. By the induction hypothesis, $F(\Y_{<n})$ and $F(Y^{n}[-n])$ are objects of $\thick(\X_B)$. Since $\thick(\X_B)$ is a triangulated subcategory of $K^b(B^{op})$, then $F(\Y)\in\thick(\X_B)$.

Second, $G(\Y)$ is an object of $K^b(\proj A)$, for any object $\Y$ of $\thick(\X_B)$. The proof is similar to the previous, just by induction on the \textit{distance} of $\Y$ to $\add \X$ (the notation is introduced in \cite[pp. 71]{H}).

Since the restrictions $F_{\proj A}: \proj A\leftrightarrow \add \X_B: {_{\add \X_B}G}$ are mutually inverse dualities. To prove that $GF$ is an identity on $K^b(\proj A)$ and $FG$ is an identity on $\thick(\X_B)$, it is also given respectively by induction on the width of complex in $K^b(\proj A)$ and by induction on the distance of complex in $\thick(\X_B)$ to $\add \X$.
\end{proof}

\begin{lemma}\mlabel{orth}
Let $A$ be a hereditary algebra. Let $\X$ be a bounded and self-orthogonal $A$-complex. Then $\emph{gl}(\emph{End}_{D^b(A)}(\X))<\infty$.
\end{lemma}
\begin{proof}
We prove this lemma by induction on the width $t(\X)$ of the complex $\X$. Let $B=\End_{D^b(A)}(\X)$.
If $t(\X)=1$, then $\X$, up to shift, is isomorphic to a partial tilting $A$-module. By \cite[Corollary III.6.5 and Proposition III.3.4]{H}, $\gl(B)<\infty$.
Assume the assertion holds for $t(\X)< n$. Let $t(\X)=n$.
We may assume without loss of generality that $\X\cong X_1[1]\oplus X_2[2]\oplus \cdots \oplus X_{n}[n]$, where $X_i$ are $A$-modules and $X_1\neq 0\neq X_n$. Let $\Y=X_1[1]\oplus X_2[2]\oplus \cdots \oplus X_{n-1}[n-1]$. Let $\frak{U}$ be the set of all morphisms in $B$ factoring through $X_n[n]$. Note that $\frak{U}$ is an idempotent ideal of $B$ and it is projective as $B$-module, moreover, $B/\frak{U}\cong\End_{D^b(A)}(\Y)$. By \cite[Corollary 5.6]{A}, $\gl(B)<\infty$ if and only if $\gl(\End_{D^b(A)}(\Y))<\infty$ and $\gl(\End_{D^b(A)}(X_n[n]))<\infty$. Note that $\Y$ and $X_n[n]$ are self-orthogonal and the widths of them are less than $n$, by the induction hypothesis, we get the assertion.
\end{proof}

\begin{theorem}\mlabel{t}
Let $A$ and $B$ be two algebras. Let $\X$ be a bounded $A$-$B$-complex. Assume $B$ is hereditary and $r_{\X}: B\rightarrow\End_{D^b(A)}(\X)^{op}$ is an isomorphism. Then $\X$ is two-sided tilting if and only if for any indecomposable projective $A$-module $Ae$,
\begin{enumerate}
\item
there is a unique $i\in \mathbb{Z}$ such that $\Hom_{D^b(A)}(Ae[i], \X)\neq 0$;
\item
there is a minimal left $\add \X$-approximation sequence
$$Ae[i]\rightarrow \X_0\rightarrow \X_1$$
of $Ae[i]$ which forms a distinguished triangle in $D^b(A)$.
\end{enumerate}
\end{theorem}
\begin{proof}
Suppose $\X$ is two-sided tilting. Then $\X$ has \ddcp. For any indecomposable projective $A$-module $Ae$ with $e$ an idempotent of $A$, the condition $(a)$ holds by Theorem \mref{charact}. By Lemma \mref{app}, there is a minimal left $\add \X$-approximation sequence $Ae[i]\stackrel{f}\rightarrow \X_0\stackrel{g}\rightarrow \X_1$ of $Ae[i]$ such that $\cone(g)\cong \RHom_{B^{op}}(\Hom_{D^b(A)}(Ae[i], \X), \X)[1]$ in $D^b(A)$. Since $\X$ is two-sided tilting and $\Hom_{D^b(A)}(Ae[i], \X)\cong eA\OT_A\X[-i]$ in $D^b(B^{op})$, we have $\cone(g)\cong Ae[i+1]$ in $D^b(A)$. This yields a distinguished triangle $Ae[i]\stackrel{f'}\rightarrow \X_0\stackrel{g}\rightarrow \X_1\rightarrow Ae[i+1]$ in $D^b(A)$. Applying $\Hom_{D^b(A)}(-, \X)$ to it, we get an exact sequence $0\rightarrow \Hom_{D^b(A)}(\X_1, \X)\rightarrow \Hom_{D^b(A)}(\X_0, \X)\rightarrow \Hom_{D^b(A)}(Ae[i], \X)\rightarrow 0$ since $\Hom_{D^b(A)}(Ae[i+1], \X)=0$ by $(a)$. So $Ae[i]\stackrel{f'}\rightarrow \X_0\stackrel{g}\rightarrow \X_1$ is a minimal left $\add \X$-approximation sequence of $Ae[i]$ in $D^b(A)$, where $f'$ is left minimal since $f$ is left minimal.

Conversely, suppose that for any indecomposable projective $A$-module $Ae$ with $e$ an idempotent of $A$,
there is a unique $i\in \mathbb{Z}$ such that $\Hom_{D^b(A)}(Ae[i], \X)\neq 0$ and
there is a minimal left $\add \X$-approximation sequence
$$Ae[i]\rightarrow \X_0\stackrel{g}\rightarrow \X_1$$
of $Ae[i]$ which forms a distinguished triangle in $D^b(A)$. Note that $\cone(g)\cong Ae[i+1]$ in $D^b(A)$, using Theorem \mref{charact}, $\X$ has \ddcp. Then by Proposition \mref{two-sided}, we only need to prove that $\X$ is isomorphic to a tilting complex in $D^b(B^{op})$.

Using Lemma \mref{app}, $Ae[i+1]\cong \cone(g)\cong \RHom_{B^{op}}(\Hom_{D^b(A)}(Ae[i], \X), \X)[1]$ in $D^b(A)$. Note that $\Hom_{D^b(A)}(Ae[i], \X)\cong eA\OT_A\X[-i]$ in $D^b(B^{op})$, then $\RHom_{B^{op}}(eA\OT_A\X, \X)\cong Ae$ in $D^b(A)$. Hence $\RHom_{B^{op}}(\X, \X)\cong A$ in $D^b(A)$ and then $\X$ is self-orthogonal in $D^b(B^{op})$. By Lemma \mref{orth}, $\gl(A)<\infty$. Now by \cite[Theorem 16]{A1}, $\gl(A\otimes_KB^{op})=\gl(A)+\gl(B)<\infty$.  Hence we may assume $\X$ is a bounded complex of projective $A\otimes_KB^{op}$-modules.  Naturally, $\X$ is an object in $K^b(\proj A)$ and also an object in $K^b(\proj B^{op})$. Consider the triangle functor
$$F:=\RHom_A(-, \X): K^b(\proj A)\rightarrow K^b(B^{op}).$$
 Since $r_{\X}: B\rightarrow\End_{D^b(A)}(\X)$ is a $B^{op}$-module isomorphism by Lemma \mref{bi} and $\End_{D^b(A)}(\X)$ is a direct summand of $F(\X)$ as $B$ is hereditary, we have $B_B$ is isomorphic to a direct summand of $F(\X)$. Moreover, by Lemma \mref{faith}, $\im(F)=\thick(\X)$ which is closed under taking direct summand. Then $B_B\in \thick(\X)$. This means that $\X$ generates $K^b(\proj B^{op})$. We are done.
\end{proof}

\begin{remark}
Without the assumption that $B$ is hereditary, Theorem \mref{t} may fail.

Let $A$ be an algebra. Let $\X$ be an $A$-$B$-bimodule $X$ such that $_AX$ is a non-projective generator of $A$ (i.e., $\add A\subsetneq\add X$) and $B=\End_{A}(X)^{op}$. Then $A\stackrel{=}\rightarrow A\rightarrow 0$ is a minimal left $\add \X$-approximation of $A$. So $\X$ satisfies the two conditions in Theorem \mref{t}. But, as $_AX$ is a non-projective generator, $|K_0(A)|\neq |K_0(B)|$, $X$ is not tilting.
\end{remark}

As a directly consequence of Theorem \mref{t}, we have

\begin{coro}\mlabel{module}
Let $A$ be an algebra and $M$ an $A$-module. Assume $\End_A(M)$ is hereditary. Then $_AM_{\End_A(M)}$ is tilting if and only if there is an exact sequence
$$0\rightarrow A\stackrel{f}\rightarrow M_0\rightarrow M_1\rightarrow 0$$
such that $M_0, M_1\in \add M$ and $f:A\rightarrow M_0$ is a minimal left $\add M$-approximation.
\end{coro}

An interesting new approach to prove the following well-known property can be given using our result.
\begin{coro}
Let $B$ be a connected hereditary algebra.  Let $A$ be another algebra such that $A$ and $B$ are derived equivalent. Then $\gl(A)\leq |K_0(B)|+1$.
\end{coro}
\begin{proof}
Let $P$ be a tilting $B^{op}$-complex such that $\End_{D^b(B^{op})}(P)\cong A$. Then there is a two-sided tilting $A$-$B$-complex $\X$ such that $\X\cong P$ in $D^b(B^{op})$. We may assume $P$ is radical \cite[pp. 112 (a)]{HX} and let $t$ be the width of $P$. Then by \cite[Section 12.5(b)]{GR}, $\gl(A)\leq \gl(B)+t-1$. Note that $B$ is connected, by \cite[Corollary 6.7.11]{Z}, $A$ is also connected. Moreover, since $\End_{D^b(B^{op})}(\X)\cong A$ and $B$ is hereditary, all the degrees of non-zero homologies of $\X$ are continuous integers. By Theorem \mref{t}, the number of the degrees of non-zero homologies of $\X$ is at most $|K_0(A)|$. Since $B$ is hereditary, $t\leq |K_0(A)|+1$. Note that by \cite[Section 12.5(d)]{GR}, $|K_0(B)|=|K_0(A)|$. Hence $\gl(A)\leq \gl(B)+t-1\leq t\leq |K_0(B)|+1$.
\end{proof}

\subsection{Characterizations II}

In this subsection, we will further investigate $A$-$B$-complexes with the derived double centraliser property and two-sided tilting $A$-$B$-complexes where algebras $A$ and $B$ are hereditary algebras. The following lemma is crucial.

\begin{lemma}\mlabel{ddcp}
Let $A$ and $B$ be hereditary algebras and $\X$ a bounded $A$-$B$-complex. Assume $r_{\X}: B\rightarrow\End_{D^b(A)}(\X)^{op}$ is an isomorphism. Let $Ae$ be an indecomposable projective $A$-module. Assume there is a unique $i\in \mathbb{Z}$ such that $\Hom_{D^b(A)}(Ae[i], \X)\neq 0$.
\begin{enumerate}[(i)]
\item
The following conditions are equivalent:
\begin{enumerate}
\item
There is a minimal left $\add \X$-approximation sequence
 $$Ae[i]\rightarrow \X_0\stackrel{g}\rightarrow \X_1$$
of $Ae[i]$ such that $Ae\cong H^{-(i+1)}(\cone(g))$ as $A$-modules.
\item
There is an exact minimal left $\add H^{-i}(\X)$-approximation sequence
$$Ae\stackrel{f}\rightarrow X_0\rightarrow X_1$$
of $Ae$ such that $\ker f\in  \add H^{-(i+1)}(\X)$.
\end{enumerate}
\item
The following conditions are equivalent:
\begin{enumerate}
\item
There is a minimal left $\add \X$-approximation sequence
$$Ae[i]\rightarrow \X_0\rightarrow \X_1$$
of $Ae[i]$ which forms a distinguished triangle in $D^b(A)$.
\item
There is an exact minimal left $\add H^{-i}(\X)$-approximation sequence
$$Ae\stackrel{f}\rightarrow X_0\rightarrow X_1\rightarrow 0$$
of $Ae$ such that $\ker f\in  \add H^{-(i+1)}(\X)$.
\end{enumerate}
\end{enumerate}
\end{lemma}
\begin{proof}
$(i)$. Suppose that there exists a minimal left $\add \X$-approximation sequence
\begin{equation}\mlabel{1}
Ae[i]\stackrel{f}\rightarrow \X_0\stackrel{g}\rightarrow \X_1
\end{equation}
of $Ae[i]$ such that $Ae\cong H^{-(i+1)}(\cone(g))$ as $A$-modules. Because $Ae[i]\stackrel{f}\rightarrow \X_0$  is a minimal left $\add \X$-approximation of $Ae[i]$ and $A$ is hereditary, $\X_0$ is concentrated in degree $-i$. Hence, observing from the distinguished triangle $\X_0\stackrel{g}\rightarrow \X_1\rightarrow \cone(g)\rightarrow \X_0[1]$, there is an exact sequence of $A$-modules
$0\rightarrow H^{-(i+1)}(\X_1)\rightarrow Ae \rightarrow H^{-i}(\X_0)\rightarrow H^{-i}(\X_1).$
We only need to prove that the sequence
\begin{equation}\mlabel{2}
Ae\rightarrow H^{-i}(\X_0)\rightarrow H^{-i}(\X_1)
\end{equation}
defined as above is a minimal left $\add H^{-i}(\X)$-approximation of $Ae$.

We claim that there exists a minimal left $\add \X$-approximation $\cone(f)\rightarrow \X_1$ of $\cone(f)$ in $D^b(A)$. Indeed,
applying $\Hom_{D^b(A)}(-, \X)$ to the distinguished triangle $Ae[i]\stackrel{f}\rightarrow \X_0\rightarrow \textmd{cone}(f)\rightarrow Ae[i+1]$, since $\Hom_{D^b(A)}(Ae[i+1], \X)=0$, we have an exact sequence
$$0\rightarrow\Hom_{D^b(A)}(\textmd{cone}(f), \X)\rightarrow\Hom_{D^b(A)}(\X_0, \X)\stackrel{m}\rightarrow\Hom_{D^b(A)}(Ae[i], \X)\rightarrow 0.$$
Since projective dimension of $B^{op}$-module $\Hom_{D^b(A)}(Ae[i], \X)$ is less or equal to $1$ and since $\Hom_{D^b(A)}(\X_1, \X)\rightarrow\Hom_{D^b(A)}(\X_0, \X)\stackrel{m}\rightarrow\Hom_{D^b(A)}(Ae[i], \X)\rightarrow 0$ is minimal presentation of $B^{op}$-module $\Hom_{D^b(A)}(Ae[i], \X)$ by Lemma \mref{left-minimal}, there is an isomorphism $\Hom_{D^b(A)}(\X_1, \X)\cong\Hom_{D^b(A)}(\cone(f'), \X)$. Note that since $r_{\X}$ is an isomorphism, applying $\Hom_{D^b(A)}(-, \X)$,
$$\Hom_{D^b(A)}(\cone(f'), \X_1)\cong\Hom_{B^{op}}(\Hom_{D^b(A)}(\X_1, \X),\Hom_{D^b(A)}(\cone(f'), \X)).$$
Again by Lemma \mref{left-minimal}, there exists a minimal left $\add \X$-approximation $\cone(f)\rightarrow \X_1$ of $\cone(f)$ in $D^b(A)$.

Applying $\Hom_{D^b(A)}(-, H^{-i}(\X)[i])$ to the sequence (\mref{1}), we have an exact sequence
\begin{equation}\mlabel{3}
\Hom_{D^b(A)}(\X_1, H^{-i}(\X)[i])\rightarrow\Hom_{D^b(A)}(\X_0, H^{-i}(\X)[i])\rightarrow\Hom_{D^b(A)}(Ae[i], H^{-i}(\X)[i])\rightarrow 0.
\end{equation}
Note that $\X_0$ is concentrated in degree $-i$, then $H^{-j}(\cone(f))=0$ when $j<i$. As above claim, we have $H^{-j}(\X_1)=0$ when $j<i$. So $\Hom_{D^b(A)}(\X_1, H^{-i}(\X)[i])\cong \Hom_{D^b(A)}(H^{-i}(\X_1)[i], H^{-i}(\X)[i])$. Therefore, the sequence (\mref{3}) induces an exact sequence
$$\Hom_A(H^{-i}(\X_1), H^{-i}(\X))\rightarrow\Hom_A(H^{-i}(\X_0), H^{-i}(\X))\rightarrow\Hom_A(Ae, H^{-i}(\X))\rightarrow 0.$$
Then we get the sequence (\mref{2}) is a left $\add H^{-i}(\X)$-approximation of $Ae$. Since the sequence (\mref{1}) is left minimal and $H^{-i}(\X)[i]$ is a direct summand of $\X$, it is not difficult to get that the sequence (\mref{2}) is left minimal.

Conversely, suppose that there is an exact minimal left $\add H^{-i}(\X)$-approximation sequence $Ae\stackrel{f}\rightarrow X_0\stackrel{g}\rightarrow X_1$ of $Ae$ such that $\ker f\in  \add H^{-(i+1)}(\X)$. Let $Ae[i]\stackrel{f[i]}\rightarrow X_0[i]\stackrel{(h_1, h_2)}\longrightarrow (\ker f)[i+1]\oplus (\textmd{coker}\, f)[i]\rightarrow Ae[i+1]$ be the distinguished triangle in $D^b(A)$ induced by $f[i]$. We will show that the following sequence
\begin{equation}\mlabel{4}
Ae[i]\stackrel{f[i]}\rightarrow X_0[i]\stackrel{(h_1, g[i])}\longrightarrow (\ker f)[i+1]\oplus X_1[i]
\end{equation}
is a minimal left $\add \X$-approximation sequence of $Ae[i]$ in $D^b(A)$ with $H^{-(i+1)}(\cone(h_1, g[i]))\cong Ae$ as $A$-modules.

First, we claim $\cone(h_1, g[i])\cong (\textmd{coker}\, g)[i]\oplus Ae[i+1]$ in $D^b(A)$, hence $Ae\cong H^{-(i+1)}(\cone(h_1, g[i]))$ as $A$-modules. Indeed, by the octahedral axiom, we get the following commutative diagram
\[
\begin{array}{c}
\xymatrix{
X_0[i]\ar[r]^{(h_1, h_2)\,\,\,\,\,\,\,\,\,\,\,\,\,\,\,\,\,\,\,\,\,\,\,\,\,\,\,\,\,\,\,\,\,\,\,\,\,\,\,\,\,}\ar[d]^{=}&(\ker f)[i+1]\oplus (\textmd{coker}\, f)[i] \ar[r]\ar[d]^{
\begin{pmatrix}
1& \\
&s
\end{pmatrix}
}& Ae[i+1] \ar[r]^{f[i+1]}\ar[d]&X_0[i+1]\ar[d]^{=}\\
X_0[i]\ar[r]^{(h_1, g[i])\,\,\,\,\,\,\,\,\,\,\,\,\,\,\,\,\,\,\,\,\,\,\,\,\,\,\,\,\,\,}& (\ker f)[i+1]\oplus X_1[i]\ar[r]\ar[d]&\cone(h_1, g[i])\ar[r]\ar[d]&X_0[i+1]\\
&(\textmd{coker}\, g)[i]\ar[r]^{=}\ar[d]&(\textmd{coker}\, g)[i]\ar[d]&\\
&(\ker f)[i+2]\oplus (\textmd{coker}\, f)[i+1]\ar[r]&Ae[i+2]&
}
\end{array}
\]
where the distinguished triangle of second column is induced by the canonical short exact sequence $0\rightarrow \textmd{coker}\, f\stackrel{s}\rightarrow X_1\rightarrow \textmd{coker}\, g\rightarrow 0$. Since $A$ is hereditary, $\Hom_{D^b(A)}((\textmd{coker}\, g)[i], Ae[i+2])=0$, then the distinguished triangle of third column is split, hence $\cone(h_1, g[i])\cong (\textmd{coker}\, g)[i]\oplus Ae[i+1]$ in $D^b(A)$.

Second, we claim that the sequence $(\mref{4})$ is a left $\add \X$-approximation sequence of $Ae[i]$ in $D^b(A)$. Indeed, it is easy to check that if we apply $\Hom_{D^b(A)}(-, H^{-i}(\X)[i])$ to the sequence $(\mref{4})$, we obtain an exact sequence since $Ae\stackrel{f}\rightarrow X_0\stackrel{g}\rightarrow X_1$ is a left $\add H^{-i}(\X)$-approximation sequence of $Ae$. Now applying $\Hom_{D^b(A)}(-, H^{-(i+1)}(\X)[i+1])$ to the sequence $(\mref{4})$, we have a sequence
$$\Hom_{D^b(A)}((\ker f)[i+1]\oplus X_1[i], H^{-(i+1)}(\X)[i+1])\stackrel{F}\rightarrow  \Hom_{D^b(A)}(X_0[i], H^{-(i+1)}(\X)[i+1])\rightarrow 0.$$
Let $t\in \Hom_{D^b(A)}(X_0[i], H^{-(i+1)}(\X)[i+1])$ and $W\rightarrow X_0[i] \stackrel{t}\rightarrow H^{-(i+1)}(\X)[i+1]\rightarrow W[1]$ the induced distinguished triangle in $D^b(A)$.
By the octahedral axiom, we get the following commutative diagram:
\[
\begin{array}{c}
\xymatrix{
W\ar[r]\ar[d]^{=}&X_0[i] \ar[r]^{t}\ar[d]^{(h_1, g[i])}& H^{-(i+1)}(\X)[i+1] \ar[r]\ar[d]^{u}&W[1]\ar[d]^{=}\\
W\ar[r]& (\ker f)[i+1]\oplus X_1[i]\ar[r]^{s}\ar[d]&V\ar[r]\ar[d]&W[1]\\
&(\textmd{coker}\, g)[i]\oplus Ae[i+1]\ar[r]^{=}\ar[d]&(\textmd{coker}\, g)[i]\oplus Ae[i+1]\ar[d]&\\
&X_0[i+1]\ar[r]&H^{-(i+1)}(\X)[i+2]&
}
\end{array}
\]
Since $A$ is hereditary and $Ae$ is projective, $\Hom_{D^b(A)}((\textmd{coker}\, g)[i]\oplus Ae[i+1], H^{-(i+1)}(\X)[i+2])=0$, then the distinguished triangle of third column is split. Let $u': V\rightarrow H^{-(i+1)}(\X)[i+1]$ such that $u'u=1$. Then $u's(h_1, g[i])=u'ut=t$. Consequently, $F$ is an epimorphism. Summarily, applying $\Hom_{D^b(A)}(-,\X)$ the sequence $(\mref{4})$ we have an exact sequence $\Hom_{D^b(A)}((\ker f)[i+1]\oplus X_1[i],\X)\rightarrow \Hom_{D^b(A)}(X_0[i],\X)\rightarrow \Hom_{D^b(A)}(Ae[i],\X)\rightarrow 0$. Therefore, the sequence $(\mref{4})$ is a left $\add \X$-approximation sequence of $Ae[i]$ in $D^b(A)$.

Third, we claim that the sequence ($\mref{4}$) is left minimal. Indeed, since $f$ and $g$ are left minimal, then so are $f[i]$ and $g[i]$, we only need to prove that $h_1: X_0[i]\rightarrow \ker f[i+1]$ is left minimal. Let $q: \ker f[i+1]\rightarrow \ker f[i+1]$ with $h_1=qh_1$. Then we have a morphism of triangles:
\[
\begin{array}{c}
\xymatrix{
Ae[i]\ar[r]^{f[i]}\ar[d]^{p}&X_0[i]\ar[r]^{(h_1, h_2)\,\,\,\,\,\,\,\,\,\,\,\,\,\,\,\,\,\,\,\,\,\,\,\,\,\,\,\,\,\,\,\,\,\,\,\,\,}\ar[d]^{=}& \ker f[i+1]\oplus (\textmd{coker}\, f)[i]\ar[r] \ar[d]^{\begin{pmatrix}
q& \\
&1
\end{pmatrix}}&Ae[i+1]\ar[d]^{p[1]} \\
Ae[i]\ar[r]^{f[i]}&X_0[i]\ar[r]^{(h_1, h_2)\,\,\,\,\,\,\,\,\,\,\,\,\,\,\,\,\,\,\,\,\,\,\,\,\,\,\,\,\,\,\,\,\,\,\,\,}& \ker f[i+1]\oplus (\textmd{coker}\, f)[i]\ar[r]&Ae[i+1]
}
\end{array}
\]
Since $Ae$ is indecomposable and $f\neq 0$, $f$ is right minimal and then so is $f[i]$.  Hence $p$ is an isomorphism. Then from the morphism of triangles, we have $q$ is an isomorphism.

$(ii)$. Suppose that there is a minimal left $\add \X$-approximation sequence
$Ae[i]\rightarrow \X_0\rightarrow \X_1$
of $Ae[i]$ which forms a distinguished triangle in $D^b(A)$. Then $\X_0$ is concentrated in degree $-i$. Hence there is an exact sequence
$$0\rightarrow H^{-(i+1)}(\X_1)\rightarrow Ae\rightarrow H^{-i}(\X_0)\rightarrow H^{-i}(\X_1)\rightarrow 0.$$
By a similar proof of $(i)$, $Ae\rightarrow H^{-i}(\X_0)\rightarrow H^{-i}(\X_1)\rightarrow 0$ is a minimal left $\add H^i(\X)$-approximation sequence of $Ae$.

Conversely, suppose there is an exact left $\add H^{-i}(\X)$-approximation sequence
$$Ae\stackrel{f}\rightarrow X_0\stackrel{g}\rightarrow X_1\rightarrow 0$$
 of $Ae$ such that $\ker f\in  \add H^{-(i+1)}(\X)$. Let $Ae[i]\stackrel{f[i]}\rightarrow X_0[i]\stackrel{(h_1, h_2)}\longrightarrow (\ker f)[i+1]\oplus (\textmd{coker}\, f)[i]\rightarrow Ae[i+1]$ be the distinguished triangle in $D^b(A)$ induced by $f[i]$. We have showed in $(a)$ that the following sequence
$$Ae[i]\stackrel{f[i]}\rightarrow X_0[i]\stackrel{(h_1, g[i])}\longrightarrow (\ker f)[i+1]\oplus X_1[i]$$
is a minimal left $\add \X$-approximation sequence of $Ae[i]$, and showed $\cone(h_1, g[i])\cong Ae[i+1]\oplus (\textmd{coker}\,g)[i]$. Note that now $\textmd{coker}\,g=0$. Then the second row of the first diagram in the proof of $(i)$ yields the desired distinguished triangle.
\end{proof}

\begin{lemma}\mlabel{act-here}
Let $A$ and $B$ be two algebras. Let $\X$ be a bounded $A$-$B$-complex. Assume $B$ is hereditary. Then for $b\in B$, $r_{\X}(b)\neq 0$ if and only if $\X\OT_BBb\neq 0$ in $D^b(A)$.
\end{lemma}
\begin{proof}
By Lemma \mref{r}, we only need to prove the sufficiency. Assume $\X\OT_BBb\neq 0$ in $D^b(A)$. Since $B$ is hereditary, the canonical epimorphism $f: B\rightarrow Bb$ is split. Let $g: Bb\rightarrow B$ such that $fg=1$. Suppose $r_{\X}(b)=0$. Then $\X\OT_Bgf=0$. Hence $\X\OT_Bf=\X\OT_Bfgf=(\X\OT_Bf)(\X\OT_Bgf)=0$. Again by Lemma $\mref{r}$, $\X\OT_BBb=0$ in $D^b(A)$. This is a contradiction.
\end{proof}

Let $A$ be a herediatry algebra and $\X$ a bounded $A$-complex. Let $B=\End_{D^b(A)}(\X)^{op}$. We may assume without loss of generality that
$$\X=X_0\oplus X_1[1]\oplus \cdots \oplus X_n[n],$$
where $X_i$ are $A$-modules and $X_0\neq 0\neq X_n$. Then each $X_i$ is endowed with a $B^{op}$-module structure given by the canonical algebra epimorphism:
$B^{op}=\End_{D^b(A)}(\X)\rightarrow \End_A(X_i).$
Hence each $X_i$ is endowed with an $A$-$B$-bimodule structure. We denote by $A$-$B$-complex
$$T(\X):=X_0\oplus X_1[1]\oplus \cdots \oplus X_n[n],$$
where $X_i$ are $A$-$B$-bimodules defined as above.
\begin{lemma}\mlabel{rTX}
 With the notation defined as above, if $B$ is hereditary, then $r_{T(\X)}$ is an isomorphism.
\end{lemma}
\begin{proof}
Since $B=\End_{D^b(A)}(\X)^{op}$, we only need to prove $r_{T(\X)}$ is injective. Since $B$ is hereditary, by Lemma \mref{act-here}, we only need to prove that for any non-zero idempotent $e\in B$, $T(\X)\OT_BBe\neq 0$ in $D^b(A)$. Note that $T(\X)\OT_BBe\cong X_0e\oplus X_1e[1]\oplus \cdots \oplus X_ne[n]$ in $D^b(A)$. Following from the construction of $B$-module structures for $X_i$, there exists $X_t$ for some $t$ such that $X_te\neq 0$. Hence the assertion holds.
\end{proof}

\begin{prop}\mlabel{dd}
Let $A$ be a hereditary algebra. Let $\X$ be a bounded $A$-complex. Assume $\End_{D^b(A)}(\X)$ is hereditary.
Then $\X$ has the derived double centraliser property if and only if for any indecomposable projective $A$-module $Ae$,
\begin{enumerate}
\item
there is a unique $i\in \mathbb{Z}$ such that
$\Hom_{D^b(A)}(Ae[i], \X)\neq 0;$
\item
there is an exact left $\emph{add}\, H^{-i}(\X)$-approximation sequence
$$Ae\stackrel{f}\rightarrow X_0\rightarrow X_1$$
of $Ae$ such that $\ker f\in  \add H^{-(i+1)}(\X)$.
\end{enumerate}
\end{prop}
\begin{proof}
Combining Theorem \mref{charact} with Lemma \mref{ddcp}$(i)$, we only need to prove the sufficiency. Suppose the conditions $(1)$ and $(2)$ hold for $\X$, then so do for $T(\X)$. By Lemma \mref{rTX}, $r_{T(\X)}$ is an isomorphism, again by Theorem \mref{charact} and Lemma \mref{ddcp}$(i)$, $T(\X)$ has \ddcp, then so does $\X$.
\end{proof}

The proof of the following proposition is similar, just by Theorem \mref{t} and Lemmas \mref{ddcp}$(ii)$ and \mref{rTX}.
\begin{prop}\mlabel{two-tilting}
Let $A$ be a hereditary algebra. Let $\X$ be a bounded $A$-complex. Assume $\End_{D^b(A)}(\X)$ is hereditary.
Then $\X$ is isomorphic to a tilting complex in $D^b(A)$ if and only if for any indecomposable projective $A$-module $Ae$,
\begin{enumerate}
\item
there is a unique $i\in \mathbb{Z}$ such that $\Hom_{D^b(A)}(Ae[i], \X)\neq 0$;
\item
there is an exact left $\emph{add}\, H^{-i}(\X)$-approximation sequence
$$Ae\stackrel{f}\rightarrow X_0\rightarrow X_1\rightarrow 0$$
 of $Ae$ such that $\ker f\in  \emph{add}\, H^{-(i+1)}(\X)$.
\end{enumerate}
\end{prop}

\begin{exam}
Let $A$ be the path algebra $KQ$ of quiver $Q: \bullet_1\rightarrow \bullet_2\rightarrow \bullet_3$.  Let $P(i)$ (resp. $I(i), S(i)$) be the indecomposable projective (indecomposable injective, simple) $A$-module corresponding to the vertexes. Let $\X=S(1)\oplus I(2)\oplus S(3)[1]$.  It is easy to see that $\End_{D^b(A)}(\X)$ is hereditary and for all integers $i$ we get
$$\Hom_{D^b(A)}(P(1)[i], \X)=0, \text{   unless   } i=0;$$
$$\Hom_{D^b(A)}(P(2)[i], \X)=0, \text{   unless   } i=0;$$
$$\Hom_{D^b(A)}(P(3)[i], \X)=0, \text{   unless   } i=1.$$
The exact left approximation sequences of projectives are given by the following maps:
$$P(1)\stackrel{f_1}\rightarrow I(2)\rightarrow 0, \text{   with   } \ker f_1 =S(3);$$
$$P(2)\stackrel{f_2}\rightarrow I(2)\rightarrow S(1)\rightarrow 0, \text{   with   } \ker f_2 =S(3);$$
$$P(3)\stackrel{=}\rightarrow S(3)\rightarrow 0.$$
By Proposition \mref{two-tilting},  $\X$ is tilting.
\end{exam}

\subsection{On homologies of complexes}

After characterized the complexes with the derived double centraliser property and two-sided tilting complexes over hereditary algebras, under the setting in the former subsection, in this subsection, we will turn our attention to the shapes of their homologies.
The following lemma is well-known, but we could not find a reference, so we provide the details.

\begin{lemma}\mlabel{sub}

\begin{enumerate}
\item
Let $A$ be a hereditary algebra and $P$ a projective $A$-module. Then $\End_A(P)$ is hereditary.
\item
Let $A$ be an algebra. Let $P$ be a projective $A$-module and $E=\End_A(P)^{op}$. Let $X$ and $Y$ be two $A$-modules. Assume for any indecomposable projective $A$-module $P'\notin \emph{add}\, P$, $\Hom_A(P', X)=0=\Hom_A(P', Y)$. Then $X$ and $Y$ have natural $E$-module structures satisfying $\Hom_E(X, Y)\cong \Hom_A(X, Y)$.
\end{enumerate}
\end{lemma}
\begin{proof}

%Let $S'$ be a simple $E$-module. Then there is a simple $A$-module $S$ such that the projective cover $P(S)\in \add P$ and $S'\cong \Hom_A(P, S)$ as $E$-modules. Applying $\Hom_A(P, -)$ to the exact sequence $0\rightarrow \rad P(S)\rightarrow P(S)\rightarrow S\rightarrow 0$, we have an exact sequence of $E$-modules $0\rightarrow \Hom_A(P, \rad P(S))\rightarrow \Hom_A(P, P(S))\rightarrow S'\rightarrow 0.$
%Note that $\Hom_A(P, P(S))\in \add _EE$, we only need to prove that $\Hom_A(P, \rad P(S))\in \add _EE$.
%Let $Q$ be the submodule of $\rad P(S)$ generated by the images of morphisms from $P$ to $\rad P(S)$. Since $A$ is hereditary, $Q$ is a projective $A$-module and $Q\in \add P$. So $\Hom_A(P, \rad P(S))\cong \Hom_A(P, Q)\in \add _EE$.

$(1)$ is trivial. 

$(2)$. We may assume that $P$ is basic. Then there is an idempotent $e\in A$ such that $P\cong Ae$ and $E\cong eAe$. Since for any projective $A$-module $P'\notin \add P$, $\Hom_A(P', X)=0=\Hom_A(P', Y)$, then $(1-e)X=0=(1-e)Y$. Hence $X$ and $Y$ are endowed with the natural $E$-module structures. Consider the map $eAe\rightarrow A/A(1-e)A, a\mapsto \bar{a}$. It is not difficult to verify that the map is an algebra epimorphism. Hence there are fully faithful embeddings:
$$eAe\textmd{-mod}\hookleftarrow A/A(1-e)A\textmd{-mod}\hookrightarrow A\textmd{-mod}.$$
So $\Hom_E(X, Y)\cong\Hom_{A/A(1-e)A}(X, Y)\cong \Hom_A(X, Y)$.
\end{proof}

\begin{prop}\mlabel{H}
Let $A$ and $B$ be two hereditary algebras. Let $\X$ be a bounded $A$-$B$-complex with \ddcp. Assume that $\X$ has non-zero homology exactly at degree in $I\subseteq \mathbb{Z}$. Then there is a decomposition $\oplus_{i\in I} P_i$ of $_AA$ with $E_i:= \emph{End}_A(P_i)^{op}$ and a decomposition $\oplus_{i\in I} P'_i$ of $B_B$ with $E'_i:=\emph{End}_{B^{op}}(P'_i)$, such that for each $i\in I$, $H^i(\X)$ endowed with natural $E_i$-$E'_i$-bimodule has the double centraliser property. Moreover, if $\X$ is two-sided tilting, the $E_i$-$E'_i$-bimodule $H^i(\X)$ is tilting.
\end{prop}
\begin{proof}
Let $\X$ be a bounded $A$-$B$-complex with \ddcp. Assume that $\X$ has non-zero homology exactly at degree in $I\subseteq \mathbb{Z}$. Using Proposition \mref{dd}, for any indecomposable projective $A$-module $Ae$,
there is a unique $i\in \mathbb{Z}$ such that $\Hom_{D^b(A)}(Ae[-i], \X)\neq 0$. Then there is a decomposition $P_i\oplus Q$ of $_AA$ such that for any $V\in \add P_i$, $\Hom_{D^b(A)}(V[-i], \X)\neq 0$, and for any $V\in \add Q$, $\Hom_{D^b(A)}(V[-i], \X)=0$. With such process, there is a decomposition $\oplus_{i\in I} P_i$ of $_AA$ such that $\Hom_{D^b(A)}(P_i[-j], \X)\neq 0$ if and only if $i=j$.
Let $E_i=\End_A(P_i)^{op}$. Then $H^i(\X)$ is endowed with natural $E_i$-module structure by Lemma \mref{sub}$(2)$. Dually, there is a decomposition $\oplus_{i\in I} P'_i$ of $B_B$ such that $\Hom_{D^b(B^{op})}(P'_i[-j], \X)\neq 0$ if and only if $i=j$;
and $H^i(\X)$ is endowed with natural $(E'_i)^{op}$-module structure where $E'_i=\End_{B^{op}}(P'_i)$. Hence $H^i(\X)$ is endowed with $E_i$-$E'_i$-bimodule structure which is induced from the $A$-$B$-bimdoule structure of $H^i(\X)$. Note that in this case, $\X\OT_B(P'_i)^*\cong H^i(\X)[-i]$ in $D^b(A)$, where $(P'_i)^*$ means the $B$-module $\Hom_{B^{op}}(P'_i, B)$. Then by Lemma \mref{sub}$(2)$,  $\End_{E_i}(H^i(\X))\cong \End_A(H^{i}(\X))\cong \End_{D^b(A)}(\X\OT_B(P'_i)^*)\cong (E'_i)^{op}$.

Using Proposition \mref{dd} again, for each $i\in I$, there is an exact left $\add _AH^i(\X)$-approximation sequence $P_i\stackrel{f}\rightarrow X_0\rightarrow X_1$ of $P_i$ such that $\ker f\in  \add H^{i-1}(\X)$. Since $\Hom_A(P_i, H^{i-1}(\X))\cong\Hom_{D^b(A)}(P_i[i-1], \X)=0$, applying $\Hom_A(P_i, -)$ to the sequence, we have an exact sequence
$$0\rightarrow E_i\rightarrow X_0\rightarrow X_1$$
of $E_i$-modules. It is an exact left $\add _{E_i}H^i(\X)$-approximation sequence of $E_i$ by Lemma \mref{sub}$(2)$. By Proposition \mref{AS}, the bimodule $_{E_i}H^i(\X)$ has the double centraliser property. Moreover, since $\End_{E_i}(H^i(\X))\cong (E'_i)^{op}$, we have $H^i(\X)$ as $E_i$-$E'_i$-bimodule has the double centraliser property.

If $\X$ is two-sided tilting, by Proposition \mref{two-tilting}, for each $i\in I$, there is an exact left $\add H^i(\X)$-approximation sequence $P_i\stackrel{f}\rightarrow X_0\rightarrow X_1\rightarrow 0$ of $P_i$ such that $\ker f\in  \add H^{i-1}(\X)$. Similar with the discussion above, applying $\Hom_A(P_i, -)$ to the sequence, we have an exact sequence
$$0\rightarrow E_i\rightarrow X_0\rightarrow X_1\rightarrow 0$$
of $E_i$-modules, which is an exact left $\add _{E_i}H^i(\X)$-approximation sequence of $E_i$. By Lemma \mref{sub}$(1)$, $E'_i=\End_{B^{op}}(P'_i)$ is hereditary. Then by Corollary \mref{module}, the bimodule $_{E_i}H^i(\X)_{E'_i}$ is tilting.
\end{proof}

Let $A$ be an algebra. Let $\X$ be a bounded $A$-complex. Denote by $H(\X)$ the $A$-complex $\oplus_{i\in \mathbb{Z}} H^i(\X)[-i]$. $\X$ is called \textit{split} if $H(\X)\cong \X$ in $D^b(A)$. As we know, all bounded $A$-complexes are split if $A$ is hereditary.
\begin{prop}\mlabel{homolo}
Let $A$ and $B$ be hereditary algebras. Let $\X$ be a bounded $A$-$B$-complex. Then $\X$ has the derived double centraliser property if and only if so does $H(\X)$. In particular, if $\X$ is two-sided tilting, then $\X$ is split as $A\otimes_KB^{op}$-complex.
\end{prop}
\begin{proof}
Let $0\neq b\in B$. Since $B$ is hereditary, assume $Bb=Be$ for some idempotent $e\in B$. Note that
$\X\OT_BBb=\X\OT_BBe=0$ in $D^b(A)$ if and only if $\X\OT_BBe$ is acyclic if and only if $H(\X)\OT_BBb=H(\X)\OT_BBe$ is acyclic if and only if $H(\X)\OT_BBb=0$ in $D^b(A)$. Then by Lemma \mref{act-here}, $r_{\X}(b)\neq 0$ if and only if $r_{H(\X)}(b)\neq 0$. Note that $\End_{D^b(A)}(\X)\cong B$ if and only if $\End_{D^b(A)}(H(\X))\cong B$. So $r_{\X}$ is an isomorphism if and only if so is $r_{H(\X)}$. Dually, $l_{\X}$ is an isomorphism if and only if so is $l_{H(\X)}$. Hence $\X$ has the derived double centraliser property if and only if so does $H(\X)$.

If $\X$ is two-sided tilting, then $\X$ has \ddcp, so does $H(\X)$. Note that $\X$ is isomorphic to a tilting complex in $D^b(A)$, then so is $H(\X)$. By Proposition \mref{two-sided}, $H(\X)$ is two-sided tilting. Note that $\X\cong H(\X)$ in $D^b(A)$, by Proposition \mref{RZ}, $\X\cong H(\X)\OT_BB_{\sigma}$ in $D^b(A\otimes_KB^{op})$ for some $\sigma\in \Aut(B)$. Hence $\X\cong \oplus_{i\in\mathbb{Z}} H^i(\X)_{\sigma}[-i]$ in $D^b(A\otimes_KB^{op})$. So $\X$ is split as $A\otimes_KB^{op}$-complex.
\end{proof}

\begin{remark}
Let $A$ and $B$ be hereditary algebras. We do not know whether all bounded $A$-$B$-complexes with the derived double centraliser property are split as $A\otimes_KB^{op}$-complex.
\end{remark}

\section{Classifying complexes of bimodules over lower triangular matrices with \ddcp}\mlabel{triangular}

In this section, we will classify complexes of bimodules over lower triangular matrix algebras with \ddcp.

Let $\Lambda_n$ be the algebra of $n\times n$ lower triangular matrices over $K$ with $n\in\mathbb{N}$. Then $\Lambda_n$ is isomorphic to the path algebra $KQ$ of quiver
$$Q: {\bullet}_1\rightarrow {\bullet}_2\rightarrow \cdots \rightarrow{\bullet}_n.$$
Let $P(i), I(i)$ and $S(i)$ be the indecomposable projective $\Lambda_n$-module, the indecomposable injective $\Lambda_n$-module and the simple $\Lambda_n$-module corresponding to the vertex $i$ of $Q$, respectively. Since $\Lambda_n$ is a Nakayama algebra, let $X(i, j)$ be the indecomposable $\Lambda_n$-module with top $S(i)$ and socle $S(j)$.

{\bf A class of $\Lambda_n$-bimodules.} Let $A=\Lambda_n$. Consider the $A$-module
$$V_m=P(m)\oplus P(m-1)\oplus \cdots \oplus P(1)\oplus I(n-1)\oplus\cdots\oplus I(m)$$
for some $1\leq m\leq n$, where $V_1={_ADA}$ and $V_n={_AA}$. Then we can verify $\End_A(V_m)^{op}\cong A$. We still denote by $V_m$ the $A$-bimodule induced by the $A$-module $V_m$. It is easy to check that the exact sequence
$$A\hookrightarrow P(m)^{n-m+1}\oplus P(m-1)\oplus \cdots \oplus P(1)\rightarrow I(n-1)\oplus \cdots\oplus I(m)$$
 is a left $\add V_m$-approximation of $A$. By Proposition \mref{AS}, $A$-bimodule $V_m$ has the double centraliser property.

{\bf A class of $\Lambda_n$-$\Lambda_n$-complexes.} Let $A=\Lambda_n$. Let $e_i\in A$ be an idempotent of $A$ such that $Ae_i\cong P(1)\oplus P(2)\oplus\cdots \oplus P(i)$ as $A$-modules, for $1\leq i\leq n-1$. Consider the $A$-$A$-complex
$$T_i=A/Ae_{n-i}A\oplus D(A/Ae_iA)[1],$$
where $D=\Hom_K(-, K)$ and $e_{n-i}=1-e_i$.

\begin{lemma}\mlabel{T_i}
The $A$-$A$-complex $T_i$, for $1\leq i\leq n-1$, is two-sided tilting.
\end{lemma}
\begin{proof}
Note that $A/Ae_{n-i}A\cong X(i, i)\oplus X(i-1, i)\oplus \cdots \oplus X(1, i)$ and $D(A/Ae_iA)\cong X(i+1, n)\oplus X(i+1, n-1)\oplus\cdots\oplus X(i+1, i+1)$ as $A$-modules. Then $\End_{D^b(A)}(\X)^{op}\cong A$. For any idempotent $0\neq e\in A$, $T_i\OT_AAe\cong (A/Ae_{n-i}A)e\oplus D(A/Ae_iA)e[1]\neq 0$ in $D^b(A)$, hence for any $0\neq a\in A$, $T_i\OT_AAa\neq 0$. By Lemma \mref{act-here}, $r_{\X}: A\rightarrow\End_{D^b(A)}(\X)^{op}$ is an isomorphism. Similarly, for any indecomposable projective $A$-module $Ae$, either $\Hom_{D^b(A)}(Ae, T_i)\neq 0$ or $\Hom_{D^b(A)}(Ae[1], T_i)\neq 0$. We point out that $Ae_{n-i}\cong A/Ae_iA$ as $A$-modules and $A/Ae_iA\cong \Lambda_{n-i}$ as algebras. Let $A/Ae_iA\rightarrow I_0\rightarrow I_1\rightarrow 0$ be a minimal injective resolution of $A/Ae_iA$-module $A/Ae_iA$. Then the induced sequence $0\rightarrow Ae_{n-i}\rightarrow I_0\rightarrow I_1\rightarrow 0$ gives an exact minimal left $\add D(A/Ae_iA)$-approximation of $Ae_{n-i}$.  A minimal left $\add A/Ae_{n-i}A$-approximation of $Ae_i$ is given by the exact sequence:
$$0\rightarrow  X(i+1, n)^i\rightarrow Ae_i\rightarrow A/Ae_{n-i}A\rightarrow 0.$$
By Proposition \mref{two-tilting}, $T_i$ is two-sided tilting.
\end{proof}

Let $A$ be an algebra. Recall from \cite{AR} that a \textit{path} from an indecomposable module $M$ to an indecomposable module $N$ in $A$-mod is a sequence of morphisms
$$M\stackrel{f_1}\rightarrow M_1\stackrel{f_2}\rightarrow M_2\rightarrow \cdots \rightarrow M_{t-1}\stackrel{f_t}\rightarrow N$$
between indecomposable modules, where $t\geq 1$ and $f_i$ is not zero and not an isomorphism. In this case, we call the \textit{length} of the path is $t$, and say the path is a \textit{zero-path }if $f_t\cdots f_2f_1=0$.

\begin{lemma}\mlabel{path}
Let $A=\Lambda_n$. Then all the paths of length $t$ in $A$-mod are zero-paths if $t\geq n$. Hence if there is a $\Lambda_n$-$\Lambda_m$-bimodule with the double centraliser property, then $m=n$.
\end{lemma}
\begin{proof}
Let $M\stackrel{f_1}\rightarrow M_1\stackrel{f_2}\rightarrow M_2\stackrel{f_3}\rightarrow \cdots \stackrel{f_{t-1}}\rightarrow M_{t-1}\stackrel{f_t}\rightarrow N$ be a path in $A$-mod with length $t$. Then it induces the sequence of morphisms:
$M\stackrel{g_1}\twoheadrightarrow \im f_1\stackrel{g_2}\twoheadrightarrow \im (f_2f_1)\stackrel{g_3}\twoheadrightarrow\cdots \stackrel{g_t}\twoheadrightarrow \im (f_t\cdots f_2f_1).$
 Let $M=X(i, j)$ and $N=X(i', j')$. Suppose that the path is nonzero and $t\geq n$. Then $\im (f_t\cdots f_2f_1)\neq 0$, hence $j\leq j'\leq i$. Note that $g_k$ is an isomorphism if only if $f_k$ is a proper monomorphism. Then there are at least $t-(j'-j)$ morphisms in the path which are proper monomorphisms. Therefore,
$i'\geq i+t-(j'-j)=t+j+(i-j')\geq t+j\geq n+j\geq n+1.$
This is a contradiction.

Let $X$ be a $\Lambda_n$-$\Lambda_m$-bimodule with the double centraliser property. Then $\End_{\Lambda_n}(X)\cong \Lambda_m$. Hence there is a non-zero path of length $m-1$ in $\Lambda_n$-mod. Then $m\leq n$. Dually, $n\leq m$. So $m=n$.
\end{proof}

\begin{lemma}\mlabel{number}
Let $A=\Lambda_n$. If a bounded $A$-$A$-complex $\X$ has \ddcp, then $\X$ has at most two non-zero homologies.
\end{lemma}
\begin{proof}
Suppose $A$-$A$-complex $\X$ has \ddcp. By definition, $A\cong \End_{D^b(A)}(\X)^{op}$. This means for each two direct summands $\X_0$ and $\X_1$ of $\X$ in $D^b(A)$, either $\Hom_{D^b(A)}(\X_0, \X_1)\neq 0$ or $\Hom_{D^b(A)}(\X_1, \X_0)\neq 0$. Suppose there exist at least three nonzero homologies of $\X$, say in degrees $i, j$ and $k$. We may assume $i<j<k$. Then
$$\Hom_{D^b(A)}(H^i(\X)[-i], H^k(\X)[-k])=0=\Hom_{D^b(A)}(H^k(\X)[-k], H^i(\X)[-i]).$$
This is a contradiction.
\end{proof}

\begin{lemma}\cite[Theorem A]{J}\mlabel{Aut}
For any $n\geq 1$, every automorphism of algebra $\Lambda_n$ is inner.
\end{lemma}

\begin{lemma}\mlabel{one-homo}
Let $A=\Lambda_n$. Then an $A$-bimodule $X$ has \dcp if and only if $X\cong V_m$ as $A$-bimodules for some $1\leq m\leq n$.
\end{lemma}
\begin{proof}
Let $X$ be an $A$-bimodule with the double centraliser property. Note that $_AX$ is faithful if and only if $P(1)\in \add X$. Since $\End_A(X)^{op}\cong A$, any two indecomposable direct summands of $_AX$ has nonzero morphism. Note that $\Hom_A(P(1), Y)=0=\Hom_A(Y, P(1))$ for any indecomposable $A$-module $Y\notin\add A\oplus DA$, we have $_AX\in \add A\oplus DA$. Then it is not difficult to verify that $X\cong V_m$ as $A$-modules for some $1\leq m\leq n$.
Now, by Lemmas \mref{KY} and \mref{Aut}, we have $X\cong V_m$ as $A$-bimodules for some $1\leq m\leq n$.
\end{proof}

\begin{lemma}\mlabel{two-homo}
Let $A=\Lambda_n$. Then a bounded $A$-$A$-complex $\X$ with exactly two nonzero homologies has the derived double centraliser property if and only if $\X$, up to shift, is isomorphic to $T_i$ in $D^b(A^e)$ for some $1\leq i\leq n-1$.
\end{lemma}
\begin{proof}
Let $\X$ be an $A$-$A$-complex with the derived double centraliser property. We may assume $\X$ concentrated in degree $0$ and $1$. Let $M=H^0(\X)$ and $N=H^1(\X)$.  By Proposition \mref{H}, there is a decomposition $P_0\oplus P_1$ of $_AA$ with $E_i:= \End_A(P_i)^{op}$ and a decomposition  $P'_0\oplus P'_1$ of $A_A$ with $E'_i:= \End_{A^{op}}(P'_i)$, such that $M$ endowed with natural $E_0$-$E'_0$-bimodule and $N$ endowed with natural $E_1$-$E'_1$-bimodule have the double centraliser property. Note that $E_i$ or $E'_i$, for $i=0$ or $1$, is isomorphic to $\Lambda_k$ for some $k$. By Lemma \mref{path}, $E_0\cong E'_0\cong \Lambda_i$ and $E_1\cong E'_1\cong \Lambda_{n-i}$, for some $1\leq i\leq n-1$.    Note that $\End_{D^b(A)}(\X)^{op}\cong A$ forces, for any direct summand $M'$ of $_AM$ and any direct summand $N'$ of $_AN$, $\Ext^1_A(M', N')\neq 0$. Then we may assume $P_0\cong P(1)\oplus P(2)\oplus\cdots \oplus P(i)$ and $P_1\cong P(i+1)\oplus P(i+2)\oplus\cdots \oplus P(n)$ as $A$-modules. By Lemma \mref{one-homo},
$$M\cong X(m, i)\oplus X(m-1, i)\oplus \cdots \oplus X(1, i)\oplus X(1, i-1)\oplus\cdots\oplus X(1, m)$$
$$N\cong X(i+l, n)\oplus X(i+l-1, n)\oplus \cdots \oplus X(i+1, n)\oplus X(i+1, n-1)\oplus\cdots\oplus X(i+1, n-i-l)$$
as $A$-modules, for some $1\leq m\leq i$ and $1\leq l\leq n-i$. Note that $\Ext^1_A(X(1, m), N)=0$ if $m\neq i$ and $\Ext^1_A(X(i+l, n), M)=0$ if $l\neq 1$. Then $m=i$ and $l=1$. Actually, now $M\cong A/Ae_{n-i}A$ and $N\cong D(A/Ae_iA)$ as $A$-modules, so $\X\cong T_i$ in $D^b(A)$. By Lemma \mref{T_i} and Proposition \mref{two-sided}, $\X$ is two-sided tilting. By Proposition \mref{RZ} and Lemma \mref{Aut}, we have $\X\cong T_i$ in $D^b(A^e)$.
\end{proof}

Combining Lemmas \mref{number}, \mref{one-homo} and \mref{two-homo}, we have
\begin{theorem}\mlabel{A_n}
 Let $A=\Lambda_n$. Let $\X$ be a bounded $A$-$A$-complex. Then $\X$ has \ddcp\, if and only if, in $D^b(A^e)$, $\X$, up to shift, is isomorphic to one of the following cases:
\begin{enumerate}
\item
 The $A$-bimodules $V_1, V_2, \cdots, V_n$;
\item
The $A$-$A$-complexes $T_1, T_2,\cdots, T_{n-1}$.
\end{enumerate}
\end{theorem}

\begin{remark}\mlabel{con}
Given two algebras $A$ and $B$ which are derived equivalent, there is a derived equivalence $F: D^b(A^e)\rightarrow D^b(B^e)$ \cite{R}. Then $F$ may not preserve \ddcp, in other words, if an $A$-$A$-complex $\X$ has \ddcp, then $F(\X )$ may have not.

Let $A=\Lambda_3$. Let $B$ be the path algebra $KQ$ of quiver $Q: \bullet\rightarrow \bullet\leftarrow \bullet$. It is known that $A$ and $B$ are derived equivalent. Recall that for algebras $C$ and $D$, $\lambda(C, D)$ is the number of all bounded $C$-$D$-complexes, up to isomorphism and shift in $D^b(C\otimes_KD^{op})$, which have \ddcp. Note that $|\textmd{Out}(B)|=2$. By Proposition \mref{out}, $\lambda(B, B)$ is even. However, by Theorem \mref{A_n}, $\lambda(A, A)=5$.
\end{remark}

\vskip 0.1in
\noindent {\bf Acknowledgements}: This work was supervised by professor Alexander Zimmermann while the author was a visiting PhD student at the Universit\'{e} de Picardie. The author would like to thank him for many helpful discussions. He also wants to thank professor Bernhard Keller for some valuable comments. Moreover, he thanks the China Scholarship Council for financial support and the LAMFA for hospitality.


\bigskip

\begin{thebibliography}{abcd}

\bibitem{A1} Auslander, M.: On the dimension of modules and algebras (III)
%(\uppercase\expandafter{\romannumeral3})
: Global dimension. Nagoya Math. J. {\bf 9}, 67-77 (1955)

%\mbibitem{AA} Auslander, M.: Representation Dimension of Artin Algebras, Queen Mary College Math. Notes, Queen Mary College, London,
%1971.
%reprinted in: Selected Works of Maurice Auslander, part 1, edited and with a foreword by Idun Reiten, Sverre O.
%Smal\o, and \O yvind Solberg, Amer. Math. Soc., Providence, 1999.

%\mbibitem{AF} Anderson, F.I., Fuller, K.R.: Rings and Categories of Modules. Graduate Texts in Mathematics,  Vol. 13, Springer-Verlag (1992)

%\mbibitem{AM} Adem, A., Milgram, R.J.: Cohomology of finite groups, Springer, Berlin, Heidelberg 1994.

\bibitem{A} Auslander, M., Platzeck, M.I., Todorov, G.: Homological theory of idempotent ideals. Trans. Amer. Math. Soc. {\bf 332}, 667--692 (1992)

\bibitem{AR} Auslander, M., Reiten, I., Smal\o, S.O.: Representation Theory of Artin Algebras. Cambridge University Press, Cambridge, 1995.

\bibitem{AS} Auslander, M., Solberg, \O.: Relative homology and representation theory III. Comm. Algebra {\bf 21}, 3081--3097 (1993)

%\mbibitem{ASS} I. Assem, D.Simson, and A.Skowro\'nski, Elements of representation theory of associative algebras, Vol. 1: Techniques of representation theory, London Mathematical Society Student Texts {\bf 65}, Cambridge University Press, New York, 2006.

%\mbibitem{BM} Bohmler, B., Marczinzik, R.: On a conjecture about Morita algebras, J. algebra {\bf 508}, 569--574 (2018)

%\mbibitem{BMR} Buan A., Marsh, R., Reineke, M., Reiten, I., Todorov, G.: Tilting theory and cluster combinatorics,
%Adv. Math. {\bf 204}, 572-¨C618 (2006)

%\mbibitem{BHM} Bergh, P.A., Han, Y., Madsen, D.: Hochschild homology and truncated cycles, Proc. Amer. Math. Soc. {\bf 140}, 1133--1139 (2011)

%\bibitem{CE} Cartan, H., Eilenberg, E.: Homological algebra. Princeton Univ. Press, Princeton, N.J., 1956.

\bibitem{CMRS} Crawley-Boevey, W., Ma, B., Rognerud, B., Sauter, J.: Combinatorics of faithfully  balanced modules. J. Combin. Theory Ser. A {\bf 182}, 105472 (2021)

%\mbibitem{CF} Colby, R.R., Fuller, K.R.: Exactness of the double dual. Proc. Amer. Math. Soc. {\bf 82}, 521--526 (1981)

%\mbibitem{CK} Chen, H.X., Koenig, S.: Ortho-symmetric modules, Gorenstain algebras, and derived equivalences. Int. Math. Res. Not. {\bf 22}, 6979--7037 (2016)

%\mbibitem{CX} Chen, H.X., Xi, C.C.: Dominant dimensions, derived equivalences and tilting modules. Israel J. Math. {\bf 215}, 349--395 (2016)


\bibitem{CPS}  Cline, E., Parshall, B.,  Scott, L.: Derived categories and Morita theory. J. Algebra {\bf 104}, 397--409 (1986)

\bibitem{FM} Fang, M., Miyachi, H.: Hochschild cohomology and dominant dimension. Trans. Amer. Math. Soc. {\bf 371}, 5267--5292 (2019)

%\mbibitem{FHK} Fang, M., Hu, W., Koenig, S.: On derive equivalences and homological dimensions, J. reine angrew. Math. preprint.

%\mbibitem{FK1} Fang, M., Koenig, S.: Schur functor and dominant dimension.  Trans. Amer. Math. Soc. {\bf 363}, 1555--1576 (2011)

%\mbibitem{FK} Fang, M., Koenig, S.: Endomorphism algebras of generators over symmetric algebras, J. Algebra {\bf 332}, 428--433 (2011)

%\mbibitem{FKY} Fang, M., Kerner, O., Yamagata, K.: Canonical bimodules and dominant dimension. Trans. Amer. Math. Soc. {\bf 370}, 847--872 (2018)

%\mbibitem{F} Fuller, K.R.: On indecomposable injectives over Artinian rings, Pacific J. Math. {\bf 29}, 115--135 (1969)

%\mbibitem{G} Geiss, C., Leclerc B., Schr\"{o}er J.: Rigid modules over preprojective algebras, Invent. Math. {\bf 165}, 589¨C-632 (2006)

\bibitem{GR} Gabriel, P., Roiter, A.V.: Representations of finite-dimensional algebras, in: Algebra. VIII, Encyclopaedia Math. Sci. {\bf 73}, Springer, Berlin, 1992.

\bibitem{Green} Green, J.A.: Polynomial Representations of $GL_n$. Second augmented edition, with an appendix on Schensted correspondence and Littelmann paths; Lecture Notes in Mathematics 830, Springer Verlag Berlin, Heidelberg 2007.

\bibitem{H} Happel, D.: Triangulated categories in the representation theory of finite-dimensional algebras. London Mathematical Society Lecture Note Series, 119. Cambridge University Press, Cambridge, 1988.

%\mbibitem{HRS} Happel, D., Reiten, I.,  Smal\o, S.: Piecewise hereditary algebras, Arch. Math. {\bf 66}, 182-186 (1996)

%\mbibitem{HZ} Happel, D., Zacharia D.: A homological characterization of piecewise hereditary algebras, Math. Z.  {\bf 260}), 177¨C185 (2008)

\bibitem{HX} Hu, W., Xi, C.C.: Derived equivalences and stable equivalences of Morita type. I. Nagoya Math. J. {\bf 200}, 107--152 (2010)
%Derived equivalences and stable equivalences of Morita type. I, Nagoya Math. J. {\bf 200}, 107¨C-152 (2010)

%\mbibitem{IS} Iyama, O., Solberg, \O.: Auslander-Gorenstein algebras and precluster tilting, Adv. Math. {\bf 326}, 200--240 (2018)

\bibitem{J}  J\o ndrup, S.: Automorphisms of upper triangular matrix rings. Arch. Math. {\bf 49}, 497--502 (1987)

\bibitem{Keller} Keller, B.: Bimodule complexes via strong homotopy actions. Algebr. Represent. Theory {\bf 3}, 357--376 (2000)

\bibitem{K} Keller, B.: Derived invariance of higher structures on the Hochschild complex. Preprint.

\bibitem{K1} Keller, B.: Deriving DG categories. Ann. Sci. \'{E}c. Norm. Sup\'{e}r. (4) {\bf 27}, 63-102 (1994)

\bibitem{KY} Kerner, O., Yamagata, K.: Morita algebras. J. Algebra  {\bf 382}, 185--202 (2013)

%\mbibitem{KSX} Koenig, S., Slung{\aa}rd, I.H., Xi, C.C.: Double centralizer properties, dominant dimension and tilting modules, J. Algebra {\bf 240}, 393--412 (2001)

%\mbibitem{L} MacLane, S.: Categories for the working mathematician. Graduate Texts in Mathematics,  Vol. 5, Springer-Verlag (1972)

%\mbibitem{Ma} Marczinzik, R.: Upper bounds for the dominant dimension of Nakayama and related algebras.  J. Algebra  {\bf 496}, 216--241 (2018)

\bibitem{MY} Miyachi, J., Yekutieli, A.: Derived Picard groups of finite-dimensional hereditary algebras. Compositio. Mathematica {\bf 129}, 341--368 (2001)

%\mbibitem{MOR} Morita, K.: Duality for modules and its applications to the theory of rings with minimum condition. Sci. Rep. Tokyo Kyoiku Daigaku Sect. A {\bf 6}, 83--142 (1958)

%\mbibitem{MO} Morita, K.: On algebras for which every faithful representation  is its own second commutator, Math. Z. {\bf 69}, 429-434 (1958)

%\mbibitem{M} Morita, K.: Duality in QF-3 Rings. Math. Z. {\bf 108}, 237--252 (1969)

%\mbibitem{MB} Mueller, B.: The classification of algebras by dominant dimensions. Canad. J. Math. {\bf 20}, 398--409 (1968)

%\mbibitem{N} Nguyen, V.C., Reiten, I., Todorov, G., Zhu, S.: Dominant dimension and tilting modules. Math. Z. {\bf 292}, 947--973 (2019)

\bibitem{R0} Rickard, J.: Morita theory for derived categories. J. Lond. Math. Soc. {\bf 39}, 436--456 (1989)
%Morita theory for derived categories. J. Lond. Math. Soc.  {\bf 39}, 436¨C-456 (1989)

%\mbibitem{R1} Rickard, J.: Derived categories and stable equivalence, J. Pure Appl. Algebra {\bf 61}, 303-317 (1989)

\bibitem{R} Rickard, J.: Derived equivalences as derived functors. J. Lond. Math. Soc. {\bf 43}, 37--48 (1991)
% {\bf 43}, 37¨C-48 (1991)

\bibitem{Ringel} Ringel, C.M.: Tame Algebras and Integral Quadratic Forms. Lecture Notes in Mathematics 1099, Springer Verlag Berlin, Heidelberg 1984.

\bibitem{RZ} Rouquier, R., Zimmermann, A.: Picard groups for derived module categories. Proc. Lond. Math. Soc. {\bf 87}, 197--225 (2003)

%\mbibitem{SY} Skowro\'nski, A., Yamagata, K.: Frobenius algebras I. EMS Textbooks in Mathematics, European Mathematical Society (EMS), Z\"urich (2011)

%\mbibitem{X} Xi, C.C.: The relative Auslander-Reiten theory of modules. preprint.
%\url{http://citeseerx.ist.psu.edu/viewdoc/download?doi=10.1.1.552.3940&rep=rep1&type=pdf}

%\mbibitem{Y} Yamagata, K.: Frobenius algebras. Handbook of Algebra,  Vol. 1, 841--887 (1996)

\bibitem{Ye} Yekutieli, A.: Dualizing complexes over noncommutative graded algebras. J. Algebra {\bf  153}, 41--84 (1992)

\bibitem{ZL} Zhang, J., Luo, Y.F.: Dominant dimension and idempotent ideals. J. Algebra {\bf 556}, 993--1017 (2020)

\bibitem{Z} Zimmermann, A.: Representation Theory; A homological algebra point of view. Algebra and Applications {\bf 19}, Springer Verlag, Cham 2014.

\end{thebibliography}
\end{document}